\documentclass[11pt,oneside]{amsart}

\usepackage{amssymb}
\usepackage{tikz}
\usepackage{tikz-qtree}
\usepackage{pdflscape}
\usepackage[UKenglish]{babel}
\usepackage{mathrsfs}
\usepackage {amsfonts}
\usepackage {german}
\usepackage[ansinew]{inputenc}
\usepackage[ps,dvips,color,all]{xy}

\usepackage{paralist}
\usepackage{pgfplots}

      \theoremstyle{plain}
      \newtheorem{theorem}{Theorem}[section]
      \newtheorem{lemma}[theorem]{Lemma}
      \newtheorem{corollary}[theorem]{Corollary}
\newtheorem{prop}[theorem]{Proposition}
\newtheorem*{con}{Conjecture}
\newtheorem*{mcon}{Markov Conjecture}

      \theoremstyle{definition}
      \newtheorem{definition}[theorem]{Definition}
      \newtheorem{ex}{Example}
      
      \theoremstyle{remark}
      \newtheorem{remark}[theorem]{Remark}

\DeclareMathOperator{\Hom}{Hom}

\DeclareMathOperator{\Coker}{Coker}

\DeclareMathOperator{\Ext}{Ext}

\DeclareMathOperator{\mo}{mod}
\DeclareMathOperator{\drep}{decrep}
\DeclareMathOperator{\dirr}{decIrr}
\DeclareMathOperator{\irr}{Irr}

\DeclareMathOperator{\rep}{rep}      
\DeclareMathOperator{\ext}{ext}

\DeclareMathOperator{\sr}{s.r.}

\DeclareMathOperator{\nil}{nil}
\DeclareMathOperator{\soc}{soc}
\DeclareMathOperator{\topp}{top}

\newcommand{\alg}{\mathbb{C}\langle \! \langle Q \rangle\!\rangle / I}
\newcommand{\com}{\mathbb{C}\langle \! \langle Q \rangle\!\rangle }
\newcommand{\M}{\mathcal{M}}
\newcommand{\nila}{\nil_A}
\newcommand{\Si}{\mathcal{S}}
\newcommand{\N}{\mathcal{N}}
\newcommand{\cl}{c_{A}}
\newcommand{\el}{e_{A}}
\newcommand{\El}{E_{A}}

\newcommand{\Homl}{\Hom_{A}}
\newcommand{\extl}{\ext^1_{A}}

\newcommand{\Dirr}{\dirr_{{\bf d,v}}(A)}
\newcommand{\srDirr}{\dirr_{{\bf d,v}}^{\sr}(A)}
\newcommand{\srdirr}{\dirr^{\sr}}

\newcommand{\Irr}{\irr_{{\bf d}}(A)}

\newcommand{\Drep}{\drep_{{\bf d,v}}(A)}
\newcommand{\Rep}{\rep_{{\bf d}}(A)}

\renewcommand{\d}{{\bf d}}
\renewcommand{\v}{{\bf v}}
\newcommand{\e}{{\bf e}}

\newcommand{\A}{\mathcal{A}}

\newcommand{\Pq}{\mathcal{P}(Q,W)}

\newcommand{\Pqb}{\mathcal{P}(Q,W')}

\newcommand{\srirr}{\irr^{\sr}(\Lambda)}
\newcommand{\srirrp}{\irr^{\sr}(\Lambda')}

      \makeatletter
      \def\@setcopyright{}
      \def\serieslogo@{}
      \makeatother
   
\begin{document}


   \author{Charlotte Ricke}
   \address{Charlotte Ricke \newline
   Mathematisches Institut \newline
   Universit\"at Bonn \newline
   Endenicher Allee 60 \newline
   53115 Bonn \newline
   Germany }
   \email{ricke@math.uni-bonn.de}

   \title{On Jacobian Algebras associated with the once-punctured torus}

   \begin{abstract}
    We consider two non-degenerate potentials for the quiver arising from the once-punctured torus, which are a natural choice to study and compare: the first is the Labardini-potential, yielding a finite-dimensional Jacobian algebra, whereas the second potential gives rise to an infinite dimensional Jacobian algebra. In this paper we determine the graph of strongly reduced components for both Jacobian algebras. Our main result is that the graph is connected in both cases. Plamondon parametrized the strongly reduced components for finite-dimensional algebras using generic $\bf{g}$-vectors. We prove that the generic $\bf{g}$-vectors of indecomposable strongly reduced components of the finite-dimensional Jacobian algebra are precisely the universal geometric coefficients for the once-punctured torus, which were determined by Reading.  
   \end{abstract}

   \subjclass[2010]{Primary 13F60; Secondary 16G20}


   \date{16th June 2014}
   \maketitle

\setcounter{tocdepth}{1}
\tableofcontents
\section{Introduction}
Cluster algebras were first introduced by Fomin and Zelevinsky in \cite{FZ1} in 2002. One of the main problems in cluster algebra theory is to find a basis of a cluster algebra with favourable properties. It is conjectured \cite[Conjecture 1.1.]{CLS} that the indecomposable strongly reduced components of Jacobian algebras parametrize such a basis of the corresponding Caldero-Chapoton algebra which sits between the cluster algebra and the upper cluster algebra \cite[Proposition 7.1.]{CLS}. 

 Let $A=\alg$ be a (possibly infinite dimensional) basic algebra.
In \cite{CLS} the authors generalise the Derksen-Weyman-Zelevinsky $E$-invariant $E_{A}(-,?)$ and the notion of strongly reduced components to arbitrary basic algebras. Denote by $\dirr (A)$ the set of irreducible components of the varieties of finite-dimensional (decorated) representations of $A$ and let $Z_1,\ldots, Z_t$ be in $\dirr (A)$.
Then by \cite[Theorem 5.11.]{CLS} the Zariski closure $\overline{Z_1 \oplus \cdots \oplus Z_t}$ is a strongly reduced irreducible component if and only if each $Z_i$ is strongly reduced and the generic $E$-invariant $E_{A}(Z_i,Z_j)=0$ for all $i \neq j$.

 Denote by $\srdirr(A)$ the subset of $\dirr(A)$ consisting of the strongly reduced components. The graph $\Gamma(A)=\Gamma(\dirr^{\sr}(A))$ of strongly reduced components has as vertices the indecomposable components in $\dirr^{\sr}(A)$, and there is an edge between -possibly equal- vertices $Z_1$ and $Z_2$ if 
$E_{A}(Z_1,Z_2)=E_{A}(Z_2,Z_1)=0$.
A \textit{component cluster} of $A$ is the set of vertices $\mathcal{U} \subset \Gamma(A)_0$ of a maximal complete subgraph of $\Gamma(A)$. 
A component cluster $\mathcal{U}$ of $A$ is $E_A$-\textit{rigid} provided that each $Z \in \mathcal{U}$ is $E_{A}$-rigid, i.e. if we have $E_{A}(Z)=0$ for all $Z \in \mathcal{U}$. The idea behind these definitions is that the Caldero-Chapoton-functions associated with a component cluster are a generalization of the clusters of a cluster algebra. 

Now, let $Q$ be the quiver 
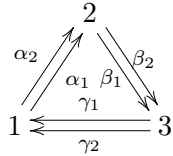
\[
\xy (0,0)*+{1}="a", (10,15)*+{2}="b", (20,0)*+{3}="c"
\ar @{>} "a";"b" < 2pt>^{\alpha_2}
\ar @{>} "a";"b" <-2pt>_{\alpha_1}
\ar @{>} "b";"c" < 2pt>_{\beta_1}
\ar @{>} "b";"c" <-2pt>^{\beta_2}
\ar @{>} "c";"a" < 2pt>^{\gamma_2}
\ar @{>} "c";"a" <-2pt>_{\gamma_1}
\endxy
\] associated with the once-punctured torus also known as the Markov quiver. There are several aspects which make this quiver an exceptional and much-studied example.
For instance, the cluster algebra $\A_Q$ is properly contained in the upper cluster algebra $\A_Q^{\text{up}}$ (\cite[Proposition 1.26]{BFZ}). Both the cluster algebra $\A_Q$ and the Jacobian algebras arising from $Q$ constitute counterexamples to many conjectures. Still, in particular its Jacobian algebras are not well understood, yet. 
  Consider the two non-degenerate potentials
\begin{align*}
W&:=\gamma_1\beta_1 \alpha_1 +\gamma_2\beta_2 \alpha_2\\
W'&:=\gamma_1\beta_1 \alpha_1 +\gamma_2\beta_2 \alpha_2- \gamma_2\beta_1\alpha_2\gamma_1\beta_2\alpha_1
\end{align*}
yielding one infinite dimensional Jacobian algebra $\Lambda=\Pq$ \cite[Example 8.6]{DWZ1} and one finite-dimensional Jacobian algebra $\Lambda'=\Pqb$  \cite[Example 8.2]{La}. 
Recently Geuenich \cite{G} proved that there are infinitely many non-degenerate potentials for $Q$ up to right equivalence. Among these, $W$ and $W'$ are a natural choice to study and compare. In this work we study the strongly reduced components of the module varieties of $\Lambda$ and $\Lambda'$. Our main tool is the truncation of basic algebras, which was introduced in \cite{CLS}. The truncation of $\Lambda$ allows us to consider $\Lambda$ as a finite-dimensional string algebra. Thus we can use the purely combinatorial description of the module categories of string algebras. Note that both algebras are tame: $\Lambda$ is gentle and $\Lambda'$ is tame by \cite[Example 6.5.3]{GLaS}.

Our first result is closely connected to $\tau$-tilting theory, which was recently introduced by Adachi, Iyama and Reiten, for finite-dimensional basic algebras. They show that $\tau$-tilting theory completes classical tilting theory from the viewpoint of mutation. However, their result does not hold for infinite dimensional algebras in general (see \cite[Example 9.3.1.]{CLS}). We show that for $\Lambda$ the mutation of support $\tau_{\Lambda}^{-1}$-tilting modules, which in fact are $E_{\Lambda}$-rigid decorated representations, is always possible and unique.  
\begin{prop}
If $Z_1,Z_2$ are indecomposable $E_{\Lambda}$-rigid components in $\srdirr(\Lambda)$ 
(resp. in $\srdirr(\Lambda')$) 
that are neighbours in $\Gamma:=\Gamma(\Lambda)$ (resp. in $\Gamma':=\Gamma(\Lambda')$), then there exist precisely two different indecomposable $E_{\Lambda}$-rigid 
components $Z_3,Z'_3 \in\srdirr(\Lambda)$ 
(resp. in $\srdirr(\Lambda')$) 
such that $\lbrace Z_1,Z_2,Z_3 \rbrace$ and $\lbrace Z_1,Z_2,Z'_3 \rbrace$ are $E_{\Lambda}$-rigid component clusters.
\end{prop}
Our main result is the connectedness of the graph of strongly reduced components for both algebras. It follows from results in \cite{P} that the full subgraph of $\Gamma'=\Gamma(\Lambda')$ on the $E_{\Lambda'}$-rigid components consists of at least two components. Thus the connectedness of $\Gamma'$ is particularly surprising. 
\begin{theorem}
\begin{compactenum}[(i)]
\item The graph $\Gamma$ is connected. The full subgraph on the $E_{\Lambda}$-rigid components is also connected.
\item The graph $\Gamma'$ is connected. The graph $\Gamma$ is a full subgraph of $\Gamma'$. The full subgraph of $\Gamma'$ on the $E_{\Lambda'}$-rigid components consists of two isomorphic components.
\end{compactenum}
\end{theorem}
For $A=\alg$ let $n$ be the number of vertices of $Q$ or equivalently the number of simple $A$-modules. Cerulli, Labardini and Schr\"oer conjecture that every component cluster has at most cardinality $n$ and that the $E_A$-rigid component clusters are precisely the component clusters of cardinality $n$. We show, that this is true for both Jacobian algebras associated with the once-punctured torus. 
\begin{theorem}
There is a bijection between the $E_{\Lambda}$-rigid indecomposable components $Z \in \dirr^{\sr}(\Lambda)$ and the non $E_{\Lambda}$-rigid indecomposable components $Z' \in \dirr^{\sr}(\Lambda)$, such that if $Z$ and $Z'$ are identified under this bijection then $\lbrace Z,Z' \rbrace $ is a component cluster and these are all non $E_{\Lambda}$-rigid component clusters.
\end{theorem}
\begin{corollary}
\begin{compactenum}[(i)]
\item The $E_{\Lambda}$-rigid component clusters of $\Lambda$ (respectively of $\Lambda'$)
are exactly the component clusters of cardinality $3$.
\item The non $E_{\Lambda}$-rigid component clusters of $\Lambda$ (respectively of $\Lambda'$) are exactly the component clusters of cardinality $2$.
\end{compactenum}
\end{corollary}
Let $Z \in \dirr^{\sr}(\Lambda)$ be a strongly reduced component. Then it is easy to see that $Z$ is also a strongly reduced component of $\Lambda'$. If $Z$ is $E_{\Lambda}$-rigid then it is also $E_{\Lambda'}$-rigid. Using a symmetry of the Auslander-Reiten translate $\tau_{\Lambda'}$ we see that there exists a unique $E_{\Lambda'}$-rigid component $\tau_{\Lambda'}Z$.  Furthermore, using the $\bf{g}$-vectors we will prove that all $E_{\Lambda'}$-rigid components arise in that manner.
The following figure shows a full subgraph in $\Gamma'$, where we left out the loops at each vertex. The full subgraph on the vertices $Z_i$ and $Z'_i$ for $i=1,2,3$ is a full subgraph of $\Gamma$.
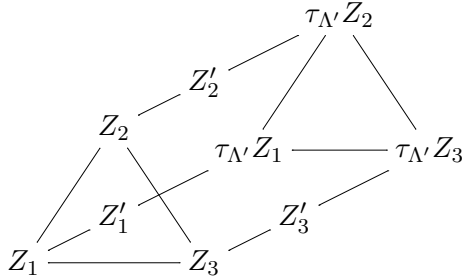
\[
\xygraph{!
{<0cm,0cm>;<0.6cm,0cm>:<0cm,0.6cm>::}
!{(0,0)}*+{Z_1}="a"
!{(2,1)}*+{Z'_1}="d" 
!{(5,2.5)}*+{\tau_{\Lambda'}Z_1}="g"
!{(2,3)}*+{Z_2}="b", 
!{(4,4)}*+{Z'_2}="e"
!{(7,5.5)}*+{\tau_{\Lambda'}Z_2}="h"
!{(4,0)}*+{Z_3}="c"
!{(6,1)}*+{Z'_3}="f"
!{(9,2.5)}*+{\tau_{\Lambda'}Z_3}="i"
"a"-"b"
"a"-"d" 
"b"-"c"
"b"-"e" 
"c"-"a"
"c"-"f"
"g"-"h"
"h"-"i"
"i"-"g"
"d"-"g"
"e"-"h"
"f"-"i"
}
\]
The paper is organised as follows: In Section 2 we give a quick overview on the background of Caldero-Chapoton algebras, strongly reduced components and string algebras. In Sections 3-6 we make all the computations to prove our results on the infinite dimensional basic algebra $\Lambda$: Section 3 contains a detailed classification of strings and bands corresponding to indecomposable strongly reduced components. In Section 4 we explicitly describe all neighbours of a fixed vertex in the graph of strongly reduced components and in Section 5 we describe its component clusters. The generic $\bf{g}$-vectors are computed in Section 6.
 Then, in Section 7 we deduce our results for the finite-dimensional algebra $\Lambda'$. 
In this context we see that the generic $\bf{g}$-vectors of indecomposable strongly reduced components of the finite-dimensional Jacobian algebra are precisely the universal geometric coefficients for the once-punctured torus. In section 8 we mention a connection between our work and the Markov Conjecture. 
\section{Background}
\subsection{Caldero-Chapoton algebras}
In \cite{FZ4} Fomin and Zelevinsky showed that the structure of cluster algebras is strongly influenced by a family of integer vectors called $\bf{g}$-\textit{vectors}, and a family of integer polynomials called \textit{F-polynomials}. For a major class of cluster algebras Derksen, Weyman and Zelevinsky found an interpretation of \textbf{g}-vectors and $F$-polynomials in terms of decorated representations of quivers with potentials. This was the motivation for defining $\bf{g}$-vectors of decorated representations of arbitrary (possibly infinite-dimensional) basic algebras and the definition of the Caldero-Chapoton algebras associated with a basic algebra. In the following we briefly recall these definitions as recently introduced by Cerulli, Labardini and Schr\"oer. We refer to \cite{CLS} for missing definitions or unexplained notation.

A \textit{quiver} $Q=(Q_0,Q_1,s,t)$ is a directed graph, where $Q_0$ denotes the finite set of \textit{vertices} and $Q_1$ the finite set of \textit{arrows}. The maps $s,t \colon  Q_1 \to Q_0$ assign a \textit{starting vertex} and a \textit{terminating vertex} to each arrow. Throughout this chapter $Q$ will be a quiver with $Q_0=\lbrace 1,\ldots, n\rbrace$ vertices.
We associate with the quiver $Q$ the matrix $B_Q=(b_{ij})\in M_n(\mathbb{Z})$, where
\[b_{ij}:=\left|\lbrace \alpha \in Q_1\mid s(\alpha)=j, \;t(\alpha)=i  \rbrace \right|-\left|\lbrace \alpha \in Q_1\mid s(\alpha)=i, \;t(\alpha)=j  \rbrace \right|.\] By definition a \textit{path of length} $k$ in $Q$ is a sequence of arrows $p=\alpha_k\ldots \alpha_2\alpha_1$ such that $s(\alpha_{i+1})=t(\alpha_i)$ for all $1 \leq i \leq k-1$.

A \textit{representation of $Q$} is a tuple  $(M_i,M_{\alpha })_{i \in Q_0, \alpha \in Q_1}$ where $M_i$ is a finite-dimensional $\mathbb{C}$-vector space and $M_{\alpha} \colon M_{s(\alpha)} \to M_{t(\alpha)}$ is a linear map. For a path $p=\alpha_k\ldots \alpha_1$ in $Q$ let $M_p:= M_{\alpha_k}\circ \ldots \circ M_{\alpha_1}$. 
The {\itshape dimension vector} of $M$ is defined to be the tuple
\[
\underline{\dim }(M):= (\dim(M_1),\ldots, \dim(M_n)).
\]
The \textit{dimension} of $M$ is $\dim (M):= \dim(M_1) + \ldots + \dim(M_n)$.

The representation $M$ is \textit{nilpotent} if there is some $N > 0$ such that $M_p=0$ for all paths $p$ in $Q$ with length($p$) $ > N$. If $M$ is a nilpotent representation the $i$th entry $\dim(M_i)$
of its dimension vector $\underline{ \dim}(M)$ equals the Jordan-H\"older multiplicity $[M : S_i]$ of $S_i$ in $M$.

A possibly infinite dimensional algebra $A$ is called \textit{basic} if $A=\alg$ for some finite quiver $Q$ and an admissible ideal $I$ of the completed path algebra $\com$. We denote by $\mo (A)$ the category of finite-dimensional (left-)modules over $A$. We can identify $\mo (A)$ with $\rep(A)$ the category of nilpotent representations of $Q$, which are annihilated by the relations in $I$. 
For $M \in \rep (A)$ we set \[\nila(M)=\min\lbrace k \in \mathbb{N}\mid M_p=0 \text{ for all paths }p \text{ of length }k \rbrace .\]
A \textit{decorated representation} of $A$ is a pair $\M =(M,V)$ where $M \in \rep(A)$ and $V=(V_1,\ldots, V_n)$ is a tuple of finite-dimensional $\mathbb{C}$-vector spaces. The category of decorated representations of $A$ will be denoted by $\drep(A)$. For each $i \in Q_0$ the \textit{simple} representation at $i$ is $\Si_i=(S_i,0)$.  The \textit{negative simple} representation at $i$ is $\Si^-_i=(0,V)$, where $V_i=\mathbb{C}$ and $V_j=0$ for all $j \neq i$.

For $\mathcal{M}=(M,V) \in \drep(A)$ we define the \textit{g-vector} of $\mathcal{M}$ to be the integral vector $g_{A}(\mathcal{M})=(g_1,\ldots,g_n) \in \mathbb{Z}^n$ where
\[g_i:=g_i(\mathcal{M}):=-\dim \Hom _{A}(S_i,M)+ \dim \Ext^1_{A}(S_i,M) + \dim(V_i).\]
For $\mathcal{M}=(M,V) ,\mathcal{N}=(N,V) \in \drep(A)$ the \textit{E-invariant} $E_{A}(\mathcal{M},\mathcal{N})$ is defined as
\[ E_{A}(\mathcal{M},\mathcal{N}):=\dim \Hom_{A}(M,N)+ \sum_{i=1}^n g_i(\mathcal{N})\dim (M_i).\]
The \textit{E-invariant} of a decorated representation $\mathcal{M}$ of $A$ is defined by
$E_{A}(\mathcal{M}):=E_{A}(\mathcal{M},\mathcal{M}).$

To every decorated representation $\M=(M,V)$ of $A$ we associate a Laurent polynomial in $\mathbb{Z}[x_1^{\pm},\ldots,x_n^{\pm}]$, the \textit{Caldero-Chapoton function} $C_A(\M)$ given by
\[C_{A}(\M):={\bf x}^{g_{A}(\M)}\sum_{\e \in \mathbb{N}^n} \chi(Gr_{\e}(M)){\bf x}^{B_Q \e}\]
where $Gr_{\e}(M)$ denotes the \textit{quiver Grassmannian}, i.e. the set of subrepresentations $U$ of $M$ with $\underline{\dim}(U)=\e$.
We denote by 
\[\mathcal{C}_{A}:=\lbrace C_{A}(\M)|\M \in \drep(A)\rbrace\]
the set of all Caldero-Chapoton functions associated to $A$. The \textit{Caldero-Chapoton algebra} $\mathcal{A}_{A}$ associated to $A$ is the $\mathbb{C}$-subalgebra of $\mathbb{C}[x_1^{\pm},\ldots,x_n^{\pm}]$ generated by $\mathcal{C}_{A}$.
\subsection{Truncation of a basic algebra}
Let $A=\alg$ be a basic algebra. For every integer $p \geq 2$ we define $I_p \subseteq A$ to be the ideal generated by all (residue classes of) paths of length $p$ in $Q$. Then the $p$\textit{-truncation} of $A$ is the finite-dimensional algebra, defined as the factor algebra
\[A_p:= A /I_p=\com/(I+\mathfrak{m}^p)\]
where $\mathfrak{m}$ denotes the \textit{arrow ideal} in $\com$, generated by all arrows in $Q$.

For a finite-dimensional representation $M$ of $A$ denote by
\[0\longrightarrow M \overset{f}{\longrightarrow} I_0^{A}(M) \overset{g}{\longrightarrow} I_1^{A}(M)\]
the minimal injective presentation of $M$ and denote by $\tau_{A_p}$ the AR translation of $A_p$. 
\begin{prop} [\cite{CLS}]\label{pro}
Let $\mathcal{M}=(M,V)$ and $\mathcal{N}=(N,V)$ be decorated representations of a basic algebra $A$ and $p > \nila(M),\nila(N)$. Then $g_{A}(\mathcal{M})=(g_1,\ldots,g_n)$ is given by
\[g_i=-[\soc(I_0^{A_p}(M)):S_i]+[\soc(I_1^{A_p}(M)):S_i]+ \dim(V_i) \]
for all $i \in Q_0$ and
\[E_{A}(\mathcal{M},\mathcal{N})=E_{A_p}(\mathcal{M},\mathcal{N})=\dim \Hom_{A_p} (\tau^{-1}_{A_p}(N),M)+ \sum_{i=1}^n \dim(W_i)\dim(M_i).\]
In particular, we have
\[\dim \Hom_{A_p} (\tau^{-1}_{A_p}(N),M)=\dim \Hom_{A_q} (\tau^{-1}_{A_q}(N),M)\]
for all $p,q > \nila(M), \nila(N)$.
\end{prop}
\begin{proof}
In \cite{CLS} this was proved with the assumption, that $p > \dim (M), \dim (N)$ instead of $p > \nila(M),\nila(N)$. Note that we always have $\dim (M) \geq \nila(M)$ (see for example the proof of \cite[Lemma 2.2.]{CLS}). However, it is easy to see, that the same proof works, using this weaker assumption. Note, that if $M \in \mo (A)$ with $p \geq \nila(M)$, then $M$ is in the image of the embedding $\mo (A_p)\rightarrow \mo (A)$. Hence, if $p \geq \nila(M),\nila(N)$ it follows that
\[\dim \Hom_{A_p}(M,N)=\dim \Hom_A(M,N).\]
Also note, that any extension of representations $M$ and $N$ of $\Lambda$ is a representation $E$ of $\Lambda$ with $\nila(E)\leq \nila(M)+\nila (N)$. Thus if $p\geq \nila(M)+\nila (N)$ we have
\[\dim \Ext^1_{A_p}(M,N)= \dim \Ext^1_A(M,N).\]
\end{proof}
Dualizing the arguments of the proof of the last proposition one obtains the following:
\begin{prop}
Let $M,N \in \mo(A)$. Then we have
\[\dim \Hom_{A_p} (N,\tau_{A_p}(M))=\dim \Hom_{A_q} (N,\tau_{A_q}(M))\]
for all $p,q > \nila(M), \nila(N)$.
\end{prop}
\begin{remark}\label{proj}
If $A$ is finite-dimensional and $M$ a non-injective representation of $A$ with minimal projective presentation 
\[P_1^{A}(M)\longrightarrow P_0^{A}(M)\longrightarrow M\longrightarrow 0\]
then
\[0\longrightarrow \tau_{A}M \overset{f}{\longrightarrow} \nu(P_1^{A}(M)) \overset{g}{\longrightarrow} \nu(P_0^{A}(M))\]
is a minimal injective presentation of $\tau_{A}M$, where $\nu \colon \mo(A) \rightarrow \mo (A)$ denotes the Nakayama functor. Then $g_{A}(\tau_{A}M)=(g_1,\ldots,g_n)$ is given by
\[g_i=-[\topp(P_1^{A}(M)):S_i]+[\topp(P_0^{A}(M)):S_i]\]
for all $i \in Q_0$.
\end{remark}
The following lemma is due to Plamondon \cite{P}.
\begin{lemma}[\cite{P}]\label{plam}
Let $A$ be a finite-dimensional Jacobian algebra. Then for $M \in \mo (A)$ we have
\[\dim \Hom_A (\tau^{-1}_{A}(M),M)=\dim \Hom_A (M,\tau_{A}(M)).\]
\end{lemma}
\subsection{Strongly reduced components}
Let $A$ be a basic algebra and $(\d,\v)\in \mathbb{N}^n\times \mathbb{N}^n$. We denote by 
\[\Irr \qquad \text{and} \qquad \Dirr \]
the irreducible components of the variety $\Rep$ consisting of representations of $A$ with dimension vector $\d$ and correspondingly of the variety $\Drep$. 
Furthermore we set
\[\irr(A)=\bigcup_{\d \in \mathbb{N}^n} \Irr \qquad \text{and} \qquad \dirr(A)=\bigcup_{(\d,\v) \in \mathbb{N}^n\times \mathbb{N}^n} \Dirr .\]
For $Z\in \Dirr$ we define the \textit{dimension vector} of $Z$ as $\underline{\dim}(Z):=(\d,\v) \in  \mathbb{N}^n\times \mathbb{N}^n$.

Let $Z,Z_1,Z_2$ be in $\dirr(A)$. Denote by $\mathcal{O}(\M)$ the isomorphism class (or $G_{\d}$-orbit) of $\M \in \Drep$. We define the following invariants
\begin{align*}
\cl(Z)&:= \min \lbrace \dim(Z)-\dim\mathcal{O}(\M)|\M \in Z \rbrace,\\
\el(Z)&:= \min \lbrace \dim\Ext^1_{A}(M,M)|\M =(M,V) \in Z \rbrace,\\
\extl(Z_1,Z_2)&:= \min \lbrace \dim\Ext^1_{A}(M_1,M_2)|\M_i=(M_i,V_i) \in Z_i, \: i=1,2\rbrace.
\end{align*}
One can show that there is a dense open subset $U$ of $Z$ (respectively of $Z_1\times Z_2$) such that $\El(\M)=\El(\N)$ for all $\M,\N \in U$ (respectively $\El(\M_1,\M_2)=\El(\N_1,\N_2)$ for all $(\M_1,\M_2),(\N_1,\N_2) \in U$). We define the \textit{generic} $E$\textit{-invariant} $\El(Z):=\El(\M)$ for $\M \in U$ (respectively $\El(Z_1,Z_2):=\El(\M_1,\M_2)$ for $(\M_1,\M_2) \in U$). By \cite[Lemma 5.2.]{CLS} we always have 
\[\cl(Z)\leq \el(Z) \leq \El(Z) \qquad \text{and}\qquad \extl(Z_1,Z_2)\leq \El(Z_1,Z_2)\]
for $Z,Z_1,Z_2$ in $\dirr(A)$.

An irreducible component $Z \in \dirr(A)$ is \textit{strongly reduced} if 
\[\cl(Z) = \el(Z) = \El(Z)\] and it is \textit{indecomposable} if it contains a dense subset of indecomposable representations. If $Z \in \Dirr$ is indecomposable then either $\d=0$ or $\v=0$.
The set of all strongly reduced components of $\Drep$ will be denoted by $\srDirr$ and we set
\[\dirr^{\sr}(A)=\bigcup _{(\d,\v)\in \mathbb{N}^n \times \mathbb{N}^n} \srDirr.\]  
Let 
\[G_{A}^{\sr}\colon \srdirr(A) \mapsto \mathbb{Z}^n\] be the map sending $Z \in \srdirr(A)$ to the generic $\bf{g}$-vector $g_{A}(Z)$ of $Z$. There is the following parametrization for strongly reduced components:
\begin{theorem}[\cite{CLS}]\label{gvectors}
For a basic algebra $A=\alg$ the following hold:
\begin{compactenum}[(i)]
\item The map
\[G_{A}^{\sr}\colon \srdirr(A) \mapsto \mathbb{Z}^n\]
is injective.
\item The following are equivalent:
\begin{compactenum}[(a)]
\item $G_{A}^{\sr}$ is surjective.
\item $\overline{A}:=\com / \overline{I}$ is finite-dimensional, where $\overline{I}$ is the $\mathfrak{m}$-adic closure of $I$.
\end{compactenum}
\end{compactenum}
\end{theorem}
\subsection{String Algebras}
A finite-dimensional basic algebra $\com/I$ is called a \textit{string algebra} if the following hold:
\begin{compactenum}[(S1)]
\item Each vertex of $Q$ is starting point of at most two arrows and end point of at most two arrows;
\item For each arrow $\beta$ there exists at most one arrow $\alpha$ such that $\beta\alpha  \notin I$ and at most one arrow $\gamma$ such that $\gamma \beta  \notin I$.
\end{compactenum}
From now on let $A = \com/I$ be a string algebra. For any arrow $\beta \in Q_1$ we define a formal inverse $\beta^-$ and we set
$Q_1^-:=\lbrace \beta^- \mid \beta \in Q_1 \rbrace$. By definition we have $(\beta^-)^-=\beta$. A \textit{string of length} $m \in \mathbb{N}$ in $A$ is a reduced word in $Q$, i.e. a sequence $C=c_m\ldots c_1$ where $c_i \in Q_1 \cup Q_1^-$ for all $i$. Furthermore $C$ does not contain a sequence of the form $\beta \beta^-$ (or $\beta^- \beta$) or any zero relations defined by $I$. A string $C=c_m\ldots c_1$ is \textit{directed} if $c_i \in Q_1$ for all $i$ or $c_i \in Q_1^-$ for all $i$. We say two strings $C$ and $C'$ are \textit{equivalent} (write $C\sim C'$) if $C=C'$ or $C^-=C'$. We choose a complete set of representatives of the set of all strings relative to $\sim$ and denote it by $\mathscr{S}_{A}$.

A string is \textit{cyclic} if it starts and ends at the same vertex. A cyclic string $B$ is called a \textit{band} if it is not directed and each power $B^m$ is a string but $B$ itself is not the power of a shorter string. If $B=c_m\ldots c_1$ is a band, then a \textit{rotation} of $B$ is a band of the form $c_i \ldots c_1c_m \ldots c_{i+1} $ for any $1 \leq i \leq m$. We say two bands $B$ and $B'$ are \textit{rotation-equivalent} and write $B\sim_r B'$ if $B'$ is a rotation of the band $B$ or $B^-$. We choose a complete set of representatives of the set of all bands relative to $\sim_r$ and denote it by $\mathscr{S}^B_{A}$.

We envision any string $C=c_m\ldots c_1$ by drawing a diagram where we draw an arrow
\[
\xygraph{!{<0cm,0cm>;<1cm,0cm>:<0cm,1cm>::}
!{(0,0)}*+{y_i}="1"
!{(1,-1)}*+{y_{i+1}}="2"
!{(3,-0.5)}*+{\text{if } c_i \in Q_1 \text{ and }}
!{(6,0)}*+{y_{i+1}}="3"
!{(5,-1)}*+{y_i}="4"
!{(8,-0.5)}*+{\text{if } c_i \in Q_1^-}
"1":"2"_{c_i}  
"3":"4"^{c_i}
}
\]
for all $1\leq i \leq n$. The \textit{string module} $M(C)$ is then given by taking the vertices $\lbrace y_1,\ldots ,y_{m+1}\rbrace $ as a $\mathbb{C}$-basis and the arrows in the diagram show, how $A$ operates on the basis vectors: if an arrow $\alpha$ appears in $C$ as $c_i$, it sends $y_i$ to $y_{i+1}$, if $\alpha^- $ appears as $c_i$, it sends $y_{i+1}$ to $y_{i}$ and otherwise it acts as zero. 

Similarly, we can define a \textit{band module} $M(B,\lambda, k)$ for every band $B$, every $\lambda \in \mathbb{C}^*$ and every integer $k\geq 1$ (see \cite{BuRi}). By excluding directed cyclic strings we make sure, that all band modules are nilpotent $A$-modules.
\begin{theorem}[\cite{BuRi}]
The string modules $M(C)$, with $C \in \mathscr{S}_{A}$ and the band modules $M(B,\lambda,k)$ with $B \in \mathscr{S}^B_{A}$, $\lambda \in \mathbb{C}^*$ and $k \in \mathbb{N}_{\geq 1}$, form a complete list of pairwise non-isomorphic representatives of indecomposable finite-dimensional $A$-modules. 
\end{theorem}
Let $C$ and $C'$ be two strings. We call a pair of triples of strings
\[(a,a')=((C_3,C_2,C_1),(C'_3,C'_2,C'_1))\]
such that $C=C_3C_2C_1$ and $C'=C'_3C'_2C'_1$ \textit{admissible} if $C_2 \sim C'_2$ and the diagrams of $C$ and $C'$ are given by
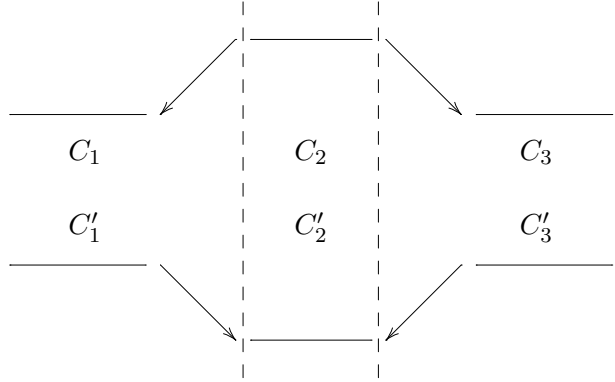
\[
\xygraph{!{<0cm,0cm>;<1cm,0cm>:<0cm,1cm>::}
!{(2,-1)}="5"
!{(3.8,-1)}="55"
!{(4,-1)}="6"
!{(5,0)}="7"
!{(5.1,0.5)}="777"
!{(5.1,-4.5)}="888"
!{(5.2,0)}="77"
!{(6.8,0)}="88"
!{(6.9,0.5)}="7777"
!{(6.9,-4.5)}="8888"
!{(7,0)}="8"
!{(8,-1)}="9"
!{(8.2,-1)}="99"
!{(10,-1)}="10"
!{(9,-1.5)}*+{C_3}="4"
!{(6,-1.5)}*+{C_2}="41"
!{(3,-1.5)}*+{C_1}="42"
!{(2,-3)}="51"
!{(3.8,-3)}="551"
!{(4,-3)}="61"
!{(5,-4)}="71"
!{(5.2,-4)}="771"
!{(6.8,-4)}="881"
!{(7,-4)}="81"
!{(8,-3)}="91"
!{(8.2,-3)}="991"
!{(10,-3)}="101"
!{(9,-2.5)}*+{C'_3}="441"
!{(6,-2.5)}*+{C'_2}="411"
!{(3,-2.5)}*+{C'_1}="421"
"5"-"55"
"7":"6"
"777":@{--}"888"
"7777":@{--}"8888"
"77"-"88"
"8":"9"
"99"-"10"  
"51"-"551"
"61":"71"
"771"-"881"
"91":"81"
"991"-"101"  
}
\]
Here it is possible that $C_i$ and $C'_i$ are strings of length zero. We denote the set of all admissible pairs for $C$ and $C'$ by $\mathscr{A}(C,C')$. The following theorem was proven in a more general set up by Crawley-Boevey \cite{CB}.
\begin{theorem}
Let $C$ and $C'$ be two strings. For any admissible pair $(a,a') \in \mathscr{A}(C,C')$ there is a canonical homomorphism $f_{(a,a')} \colon M(C) \to M(C')$. Moreover the set
\[\lbrace f_{(a,a')}\mid (a,a') \in \mathscr{A}(C,C') \rbrace \]
is a basis of $\Homl (M(C),M(C'))$.
\end{theorem}
This description of homomorphisms can be transferred to band modules and one should keep the same picture in mind. For a band $B$ we consider admissible pairs of the string $B^k$ for any positive integer $k$. However we do not want to count the same pair twice in different copies of $B$. 

Butler and Ringel gave a description of all the Auslander-Reiten sequences containing string modules. We will now apply their results to the case where $\Lambda$ is the infinite dimensional Jacobian algebra $\Pq$ as in the introduction. However, the $p$-truncation $\Lambda_p$ is finite-dimensional and thus a string algebra for every integer $p \geq 2$. Hence, if we set
\[ \mathscr{S}_{\Lambda}=\bigcup_{p\geq 2} \mathscr{S}_{\Lambda_p} \qquad \text{and} \qquad \mathscr{S}^B_{\Lambda}=\bigcup_{p\geq 2} \mathscr{S}^B_{\Lambda_p}\]
then the string modules $M(C)$ with $C \in \mathscr{S}_{\Lambda}$ and the band modules $M(B,\lambda,k)$ with $B \in \mathscr{S}^B_{\Lambda}$, $\lambda \in \mathbb{C}^*$ and $k \in \mathbb{N}_{\geq 1}$, form a complete list of pairwise non-isomorphic representatives of indecomposable $\Lambda$-modules. 
\begin{theorem}
Let $C$ be a string of length $m$ in  $\mathscr{S}_{\Lambda}$ and fix $p > m$. Then we have:
\begin{compactenum}[(i)]
\item There are unique arrows $\beta_1, \beta_2$ and directed strings $D_1,D_2$ of length $p$ such that
$\tau_{\Lambda_p}^{-1}(C):= {}_hC_h:=D_2\beta_2^-C\beta_1D_1^-$
is a string in $\mathscr{S}_{\Lambda_p}$ and 
\[\tau_{\Lambda_p}^{-1}(M(C)) \cong M(\tau_{\Lambda_p}^{-1}(C))\]
and $D_2\beta_2^-$ and $\beta_1D_1^-$ will be referred to as hooks.
\item There are unique arrows $\gamma_1, \gamma_2$ and directed strings $E_1,E_2$ of length $p$ such that
$\tau_{\Lambda_p}(C):= {}_cC_c:=E_2^-\gamma_2C\gamma_1^-E_1$
is a string in $\mathscr{S}_{\Lambda_p}$ and 
\[\tau_{\Lambda_p}(M(C)) \cong M(\tau_{\Lambda_p}(C))\]
and $E_2^-\gamma_2$ and $\gamma_1^-E_1$ will be referred to as cohooks.
\item If $C$ is in fact a band $B$, then
\[\tau_{\Lambda_p}^{-1}(M(B,\lambda,k)) \cong M(B,\lambda,k)\] for all $\lambda \in \mathbb{C}^*$ and all $k\in \mathbb{N}_{\geq 1}$. 
\end{compactenum}
\end{theorem}
\begin{remark}
Let $C,C'$ be strings of length $m$ and $m'$ respectively. Set $E_{\Lambda}(C,C'):=E_{\Lambda}(M(C),M(C'))$. If we fix $p> \max \lbrace m+1, m+1'\rbrace$, then we can compute
\[E_{\Lambda}(C,C')=\dim \Hom_{\Lambda_p}(M(C'),\tau_{\Lambda_p}^{-1}(M(C)))=\left|\mathscr{A}(\tau_{\Lambda_p}^{-1}(C'),C)\right|.\]
\end{remark}

We call a string $E_{\Lambda}$\textit{-rigid} if the corresponding string module is $E_{\Lambda}$-rigid.
We call a band $B$ \textit{strongly reduced} if the band module $M:=M(B,\lambda,1)$ satisfies $E_{\Lambda}(M)=1$.

Let $A$ be a tame algebra and $Z$ an indecomposable irreducible component. Then it is known, that $Z$ is either the closure of the orbit of an indecomposable module or the closure of the union of the orbits of a dense 1-parameter family of indecomposable modules. By \cite[Lemma 5.6.]{CLS}, we know that it is enough to classify the strongly reduced components of $\drep(\Lambda_p)$ for all $p\geq 2$. Hence we can always choose $p$ big enough and assume that $\Lambda=\Lambda_p$ is a string algebra. This yields the following theorem.
\begin{theorem}\label{bi}
Let $C$ be an $E_{\Lambda}$-rigid string and $B$ a strongly reduced band. Then \[Z= \overline{\mathcal{O}(M(C))}\qquad \text{and} \qquad Z'=\overline{\bigcup_{\lambda \in \mathbb{C}^*}\mathcal{O}(M(B,\lambda,1))}\]
 are in $\srirr$ indecomposable, where $Z$ is $E_{\Lambda}$-rigid and $Z'$ is not. On the other hand any indecomposable strongly reduced component in $\srirr$ arises as one of the above.
\end{theorem}
\section{Classification of indecomposable strongly reduced components}
\subsection{The $E_{\Lambda}$-rigid components}
By Theorem \ref{bi} it is enough to give a complete description of the $E_{\Lambda}$-rigid strings in order to classify the indecomposable $E_{\Lambda}$-rigid components in $\srdirr(\Lambda)$. In order to do so we first need to introduce some notation.

For $x_1 \in \lbrace \alpha_1, \beta_1, \gamma_1\rbrace $ we denote by $(x_1:0)$ the string of length 0 at the vertex $s(x_1)$. Now let $x_1 \in \lbrace \alpha_1^{\pm}, \beta_1^{\pm}, \gamma_1^{\pm}\rbrace $. By a sequence  $(x_1:a_1,a_2,\ldots ,a_n)$
with $a_i \in \mathbb{N}_{\geq 1}$ for all $1\leq i\leq n$ we denote the string 
\[C=(x_2^- x_1)^{a_n} y_2 y_1^-(x_2^- x_1)^{a_{n-1}} \ldots (x_2^- x_1)^{a_2}y_2y_1^-(x_2^- x_1)^{a_1}\]
where $x_2\in \lbrace \alpha_2^{\pm}, \beta_2^{\pm}, \gamma_2^{\pm}\rbrace$ such that $s(x_2)=s(x_1)$ and $y_1\in \lbrace \alpha_1^{\pm}, \beta_1^{\pm}, \gamma_1^{\pm}\rbrace\setminus \lbrace x_1 \rbrace$.
If $x_1$ is in $Q_1$ the diagram of $C$ is given by:
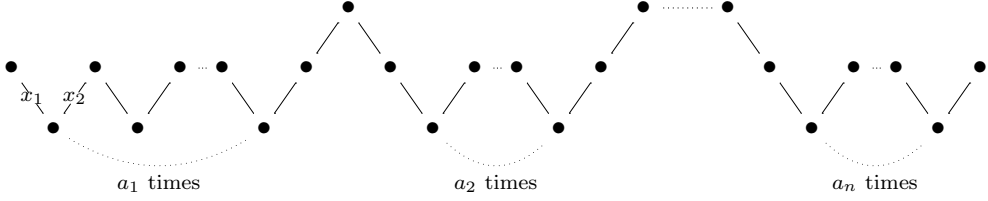
\[
\xygraph{!{<0cm,0cm>;<0.56cm,0cm>:<0cm,0.8cm>::}
!{(0,0)}*+{\bullet}="1"
!{(1,-1)}*+{\bullet}="2"
!{(2,0)}*+{\bullet}="3"
!{(3,-1)}*+{\bullet}="32"
!{(4,0)}*+{\bullet}="33"
!{(5,0)}*+{\bullet}="4"
!{(6,-1)}*+{\bullet}="5"
!{(7,0)}*+{\bullet}="6"
!{(8,1)}*+{\bullet}="7"
!{(9,0)}*+{\bullet}="8"
!{(10,-1)}*+{\bullet}="9"
!{(11,0)}*+{\bullet}="10"
!{(12,0)}*+{\bullet}="11"
!{(13,-1)}*+{\bullet}="12"
!{(14,0)}*+{\bullet}="13"
!{(15,1)}*+{\bullet}="14"
!{(17,1)}*+{\bullet}="17"
!{(18,0)}*+{\bullet}="18"
!{(19,-1)}*+{\bullet}="19"
!{(20,0)}*+{\bullet}="20"
!{(21,0)}*+{\bullet}="21"
!{(22,-1)}*+{\bullet}="22"
!{(23,0)}*+{\bullet}="23"
"1"-"2"|-{x_1}
"2"-"3"|-{x_2}
"3"-"32"
"32"-"33"
"2":@/_0.5cm/@{.}_{a_1 \text{ times}}"5"
"33":@{.}"4"
"4"-"5"
"5"-"6"
"6"-"7"
"7"-"8"
"8"-"9"
"9":@/_0.5cm/@{.}_{a_2 \text{ times}}"12"
"9"-"10"
"10":@{.}"11"
"11"-"12"
"12"-"13"
"13"-"14"
"14":@{.}"17"
"17"-"18"
"18"-"19"
"19"-"20"
"19":@/_0.5cm/@{.}_{a_n \text{ times}}"22"
"20":@{.}"21"
"21"-"22"
"22"-"23"
}
\]
If $x_1$ is in $Q_1^-$ the diagram of $C$ is given by:
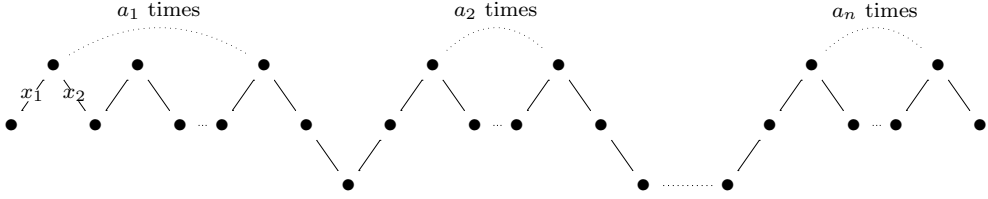
\[
\xygraph{!{<0cm,0cm>;<0.56cm,0cm>:<0cm,-0.8cm>::}
!{(0,0)}*+{\bullet}="1"
!{(1,-1)}*+{\bullet}="2"
!{(2,0)}*+{\bullet}="3"
!{(3,-1)}*+{\bullet}="32"
!{(4,0)}*+{\bullet}="33"
!{(5,0)}*+{\bullet}="4"
!{(6,-1)}*+{\bullet}="5"
!{(7,0)}*+{\bullet}="6"
!{(8,1)}*+{\bullet}="7"
!{(9,0)}*+{\bullet}="8"
!{(10,-1)}*+{\bullet}="9"
!{(11,0)}*+{\bullet}="10"
!{(12,0)}*+{\bullet}="11"
!{(13,-1)}*+{\bullet}="12"
!{(14,0)}*+{\bullet}="13"
!{(15,1)}*+{\bullet}="14"
!{(17,1)}*+{\bullet}="17"
!{(18,0)}*+{\bullet}="18"
!{(19,-1)}*+{\bullet}="19"
!{(20,0)}*+{\bullet}="20"
!{(21,0)}*+{\bullet}="21"
!{(22,-1)}*+{\bullet}="22"
!{(23,0)}*+{\bullet}="23"
"1"-"2"|-{x_1}
"2"-"3"|-{x_2}
"3"-"32"
"32"-"33"
"2":@/^0.5cm/@{.}^{a_1 \text{ times}}"5"
"33":@{.}"4"
"4"-"5"
"5"-"6"
"6"-"7"
"7"-"8"
"8"-"9"
"9":@/^0.5cm/@{.}^{a_2 \text{ times}}"12"
"9"-"10"
"10":@{.}"11"
"11"-"12"
"12"-"13"
"13"-"14"
"14":@{.}"17"
"17"-"18"
"18"-"19"
"19"-"20"
"19":@/^0.5cm/@{.}^{a_n \text{ times}}"22"
"20":@{.}"21"
"21"-"22"
"22"-"23"
}
\]
Note that identifying a sequence $(x_1:a_1,\ldots a_n)$ with a string $C$ is an injective mapping. Furthermore if $C$ is given by the sequence $(x_1:a_1,\ldots ,a_n)$ then the inverse string $C^{-}$ is given by the sequence $(x_2:a_n,\ldots,a_1)$. Since we are only interested in the strings up to the equivalence relation $\sim$ we restrict to sequences with $x_1\in \lbrace \alpha_1^{\pm}, \beta_1^{\pm}, \gamma_1^{\pm}\rbrace$.
\begin{ex}
We would like to clarify the above notation with two examples. The sequence $(\beta_1:3)$ corresponds to the string with diagram 
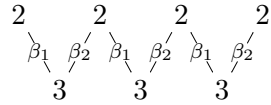
\[
\xygraph{!{<0cm,0cm>;<0.54cm,0cm>:<0cm,1cm>::}
!{(0,0)}*+{2}="1"
!{(1,-1)}*+{3}="2"
!{(2,0)}*+{2}="3"
!{(3,-1)}*+{3}="32"
!{(4,0)}*+{2}="33"
!{(5,-1)}*+{3}="5"
!{(6,0)}*+{2}="6"
"1"-"2"|-{\beta_1}
"3"-"2"|-{\beta_2}
"3"-"32"|-{\beta_1}
"33"-"32"|-{\beta_2}
"33"-"5"|-{\beta_1}
"5"-"6"|-{\beta_2}
}
\]
and the string given by the sequence $(\beta_1:3,3,3)$ has diagram
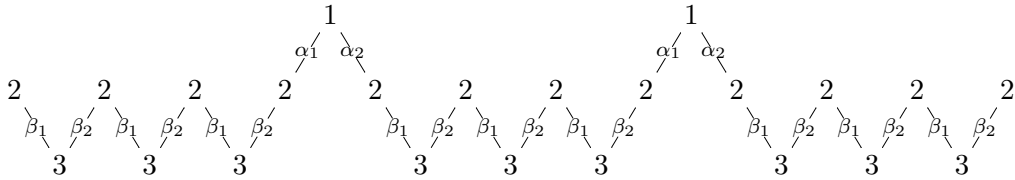
\[
\xygraph{!{<0cm,0cm>;<0.6cm,0cm>:<0cm,1cm>::}
!{(0,0)}*+{2}="1"
!{(1,-1)}*+{3}="2"
!{(2,0)}*+{2}="3"
!{(3,-1)}*+{3}="32"
!{(4,0)}*+{2}="33"
!{(5,-1)}*+{3}="5"
!{(6,0)}*+{2}="6"
!{(7,1)}*+{1}="7"
!{(8,0)}*+{2}="8"
!{(9,-1)}*+{3}="9"
!{(10,0)}*+{2}="10"
!{(11,-1)}*+{3}="30"
!{(12,0)}*+{2}="11"
!{(13,-1)}*+{3}="12"
!{(14,0)}*+{2}="13"
!{(15,1)}*+{1}="14"
!{(16,0)}*+{2}="18"
!{(17,-1)}*+{3}="19"
!{(18,0)}*+{2}="20"
!{(19,-1)}*+{3}="31"
!{(20,0)}*+{2}="21"
!{(21,-1)}*+{3}="22"
!{(22,0)}*+{2}="23"
"1"-"2"|-{\beta_1}
"3"-"2"|-{\beta_2}
"3"-"32"|-{\beta_1}
"33"-"32"|-{\beta_2}
"33"-"5"|-{\beta_1}
"5"-"6"|-{\beta_2}
"6"-"7"|-{\alpha_1}
"7"-"8"|-{\alpha_2}
"8"-"9"|-{\beta_1}
"9"-"10"|-{\beta_2}
"10"-"30"|-{\beta_1}
"11"-"30"|-{\beta_2}
"11"-"12"|-{\beta_1}
"12"-"13"|-{\beta_2}
"13"-"14"|-{\alpha_1}
"14"-"18"|-{\alpha_2}
"18"-"19"|-{\beta_1}
"19"-"20"|-{\beta_2}
"20"-"31"|-{\beta_1}
"21"-"31"|-{\beta_2}
"21"-"22"|-{\beta_1}
"22"-"23"|-{\beta_2}
}
\]
\end{ex}
In the proofs of this chapter we will often work with diagrams of strings.  Let $C$ be a string and $\tau_{\Lambda}^{-1}(C)={}_h C_h$. The hooks of ${}_h C_h$ will be displayed in green. If there is an admissible pair $((C'_3,C'_2,C'_1),(C_3,C_2,C_1))$ in $\mathscr{A}(\tau_{\Lambda}^{-1}(C),C)$ we distinguish the members of the middle terms $C'_2$ and $C_2$ with arrows of the type $\Longrightarrow $.
\begin{lemma}\label{form}
Up to equivalence any $E_{\Lambda}$-rigid string $C$ is of the form $(x_1:a_1,\ldots ,a_n)$.
\end{lemma}
\begin{proof}
It is easy to see that the strings of length zero corresponding to the simple representations are $E_{\Lambda}$-rigid and these correspond to sequences $(x_1:0)$. Now assume that $C=c_m\ldots c_1$ with $m\geq 1$. 
 We consider the case that $c_1$ is in $Q_1$ (the case $c_1 \in Q_1^-$ is similar) and because of the symmetry of $Q$ we can restrict to the case $c_1=\alpha_1$. 

Let $1 \leq i \leq m$ be maximal such that $c_i\ldots c_{1} $ is in $ Q_1$. We have to show that $i=1$:

Case 1: Let $i \equiv 0 \text{ mod }  3$.
If $m>i$ we know that $c_{i+1}$ is in $Q_1^-$. Thus the left side of the diagram of $\tau_{\Lambda}^{-1}(C)$ is 
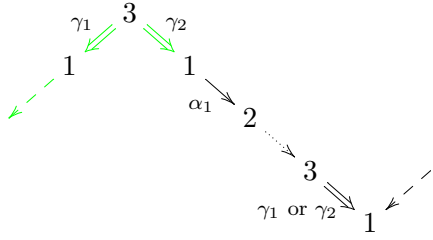
\[
\xygraph{!{<0cm,0cm>;<0.8cm,0cm>:<0cm,0.7cm>::}
!{(8,-2)}="8"
!{(9,-1)}*+{1}="9"
!{(10,0)}*+{3}="10"
!{(11,-1)}*+{1}="12"
!{(12,-2)}*+{2}="13"
!{(13,-3)}*+{3}="14"
!{(14,-4)}*+{1}="15"
!{(15,-3)}="16"
"9":@{-->}@[green]"8"
"10":@{=>}@[green]"9"_{\gamma_1}
"10":@{=>}@[green]"12"^{\gamma_2}
"12":"13"_{\alpha_1}
"13":@{.>}"14"
"14":@{=>}"15"_{\gamma_1 \text{ or } \gamma_2}
"16":@{-->}"15"
}
\] and we see there is an admissible pair in $\mathscr{A}(\tau_{\Lambda}^{-1}(C),C)$.

Case 2: Let $i \equiv 1 \text{ mod }  3$ and $ i \neq 1$. Then the left side of the diagram of $\tau_{\Lambda}^{-1}(C)$ is 
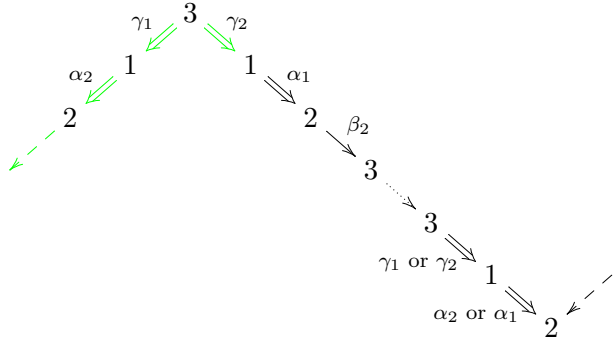
\[
\xygraph{!{<0cm,0cm>;<0.8cm,0cm>:<0cm,0.7cm>::}
!{(7,-3)}="20"
!{(8,-2)}*+{2}="8"
!{(9,-1)}*+{1}="9"
!{(10,0)}*+{3}="10"
!{(11,-1)}*+{1}="12"
!{(12,-2)}*+{2}="13"
!{(13,-3)}*+{3}="14"
!{(14,-4)}*+{3}="15"
!{(15,-5)}*+{1}="16"
!{(16,-6)}*+{2}="17"
!{(17,-5)}="18"
"8":@{-->}@[green]"20"
"9":@{=>}@[green]"8"_{\alpha_2}
"10":@{=>}@[green]"9"_{\gamma_1}
"10":@{=>}@[green]"12"^{\gamma_2}
"12":@{=>}"13"^{\alpha_1}
"13":"14"^{\beta_2}
"14":@{.>}"15"
"15":@{=>}"16"_{\gamma_1 \text{ or } \gamma_2}
"16":@{=>}"17"_{\alpha_2 \text{ or } \alpha_1}
"18":@{-->}"17"
}
\]and we see there is a non-trivial homomorphism from $\tau^{-1}_{\Lambda}(M(C))$ to $M(C)$. 

Case 3: Let $i \equiv 2 \text{ mod }  3$. Then  the left side of the diagram of $\tau_{\Lambda}^{-1}(C)$ is
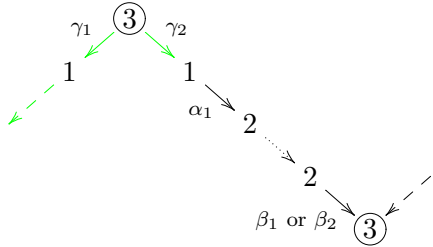
\[
\xygraph{!{<0cm,0cm>;<0.8cm,0cm>:<0cm,0.7cm>::}
!{(8,-2)}="8"
!{(9,-1)}*+{1}="9"
!{(10,0)}*+{3}="10" *\cir{} 
!{(11,-1)}*+{1}="12"
!{(12,-2)}*+{2}="13"
!{(13,-3)}*+{2}="14"
!{(14,-4)}*+{3}="15" *\cir{}
!{(15,-3)}="16"
"9":@[green]@{-->}"8"
"10":@[green]"9"_{\gamma_1}
"10":@[green]"12"^{\gamma_2}
"12":"13"_{\alpha_1}
"13":@{.>}"14"
"14":"15"_{\beta_1 \text{ or } \beta_2}
"16":@{-->}"15"
}
\]and we see there is a non-trivial homomorphism from $\tau^{-1}_{\Lambda}(M(C))$ to $M(C)$. Also note, that if $c_1= \alpha_1$ then the vertex $3$ may not appear in the socle of the diagram of $C$.
 
So now we know that if $C=c_m\ldots c_1$ is $E_{\Lambda}$-rigid and $c_1=\alpha_1$, then we must have  $c_2=\alpha_2^-$ (or more general, if $c_1=x_1$ then $c_2=x_2^-$). Note that since $C^{-}=c^-_1\ldots c^-_m$ is also $E_{\Lambda}$-rigid, we find that if $c_m$ is in $Q_1$ then $c_{m-1}$ is in $Q_1^-$ and vice versa. Another easy case-by-case study shows that $c_m=\alpha_2^-$ (or more general, if $c_1=x_1$ then $c_m=x_2^-$).

Next we show the following: If $C=c_m\ldots c_{j+i}\ldots c_{j+1}c_{j}\ldots c_1$ such that $c_{j+i},\ldots,c_{j+1}$ are in $Q_1$ (or in $Q_1^-$ respectively) and such that $c_{j+i+1}$ is in $Q_1^-$ (or in $Q_1$ respectively) then $i\leq 2$. 
We assume the above claim does not hold and again show by a case-by-case study, that there is an admissible pair for $\tau_{\Lambda}^{-1}(C)$ and $C$. Consider the case that $c_{j+i}\ldots c_{j+1}$ are in $Q_1$ (the case $c_{j+i}\ldots c_{j+1}$ in $Q_1^-$ is symmetric)  with $i\geq 3$ and such that ...

Case 1: ... $t(c_{j+i})=1$. Hence $c_{i+j}=\gamma_1$ or $c_{i+j}=\gamma_2$ and we find a non-trivial homomorphism depicted in the diagram of $\tau_{\Lambda}^{-1}(C)$:
\[
\xygraph{!{<0cm,0cm>;<0.6cm,0cm>:<0cm,0.9cm>::}
!{(0,-3)}="2"
!{(1,-2)}*+{2}="3"
!{(2,-1)}*+{1}="4"
!{(3,0)}*+{3}="5"
!{(4,-1)}*+{1}="6"
!{(5,-2)}*+{2}="7"
!{(6,-1)}*+{1}="8"
!{(8,-1)}="9"
!{(8.5,0)}*+{1}="10"  
!{(9.5,-1)}*+{2}="12"
!{(10.5,-2)}*+{3}="13"
!{(11.5,-3)}*+{1}="14"
!{(12.5,-2)}="15"
!{(13,-1)}="30"
!{(15,-1)}*+{1}="22"
!{(16,-2)}*+{2}="23"
!{(17,-1)}*+{1}="24"
!{(18,0)}*+{3}="25"
!{(19,-1)}*+{1}="26"
!{(20,-2)}*+{2}="27"
!{(21,-3)}="28"
"3":@[green]@{-->}"2"
"4":@[green]"3"_{\alpha_2}
"5":@[green]@{=>}"4"_{\gamma_1}
"5":@[green]"6"^{\gamma_2}
"6":"7"^{\alpha_1}
"8":"7"^{\alpha_2}
"8":@{.}"9"
"10":"12"
"12":"13"
"13":@{=>}"14"_{\gamma_1 \text{ or } \gamma_2}
"15":@{-->}"14"
"30":@{.}"22"
"22":"23"_{\alpha_1}
"24":"23"_{\alpha_2}
"25":@[green]"24"_{\gamma_1}
"25":@[green]@{=>}"26"^{\gamma_2}
"26":@[green]"27"^{\alpha_1}
"27":@[green]@{-->}"28"
}
\]

Case 2: ... $t(c_{j+i})=2$. Hence  $c_{j+i}c_{j+i-1}=\alpha_2\gamma_1$ or $c_{j+i}c_{j+i-1}=\gamma_2,\alpha_1$ and we find a non-trivial homomorphism depicted in the diagram of $\tau_{\Lambda}^{-1}(C)$:
\[
\xygraph{!{<0cm,0cm>;<0.6cm,0cm>:<0cm,0.9cm>::}
!{(0,-3)}="2"
!{(1,-2)}*+{2}="3"
!{(2,-1)}*+{1}="4"
!{(3,0)}*+{3}="5"
!{(4,-1)}*+{1}="6"
!{(5,-2)}*+{2}="7"
!{(6,-1)}*+{1}="8"
!{(8,-1)}="9"
!{(8.5,0)}*+{2}="10"  
!{(9.5,-1)}*+{3}="12"
!{(10.5,-2)}*+{1}="13"
!{(11.5,-3)}*+{2}="14"
!{(12.5,-2)}="15"
!{(13,-1)}="30"
!{(15,-1)}*+{1}="22"
!{(16,-2)}*+{2}="23"
!{(17,-1)}*+{1}="24"
!{(18,0)}*+{3}="25"
!{(19,-1)}*+{1}="26"
!{(20,-2)}*+{2}="27"
!{(21,-3)}="28"
"3":@[green]@{-->}"2"
"4":@[green]@{=>}"3"_{\alpha_2}
"5":@[green]@{=>}"4"_{\gamma_1}
"5":@[green]"6"^{\gamma_2}
"6":"7"^{\alpha_1}
"8":"7"^{\alpha_2}
"8":@{.}"9"
"10":"12"
"12":@{=>}"13"_{\gamma_1 \text{ or } \gamma_2}
"13":@{=>}"14"_{\alpha_2 \text{ or } \alpha_1}
"15":@{-->}"14"
"30":@{.}"22"
"22":"23"_{\alpha_1}
"24":"23"_{\alpha_2}
"25":@[green]"24"_{\gamma_1}
"25":@[green]@{=>}"26"^{\gamma_2}
"26":@[green]@{=>}"27"^{\alpha_1}
"27":@[green]@{-->}"28"
}
\]

Case 3: ... $t(c_{j+i})=3$. Then the vertex $3$ is in the socle of the diagram of $C$ which was excluded before.
 
Now assume that $C=\alpha_2^-\alpha_1c_{n-2}\ldots c_3\alpha_2^- \alpha_1$ and that the vertex $1$ appears in the socle of the diagram of $C$. We look at the part of the diagram, where this happens for the first time. Since the vertex $3$ cannot appear in the socle and the vertex $2$ cannot appear in the top of the diagram of $C$, the only possibility for this yields a non-trivial homomorphism, depicted below in the diagram of $\tau_{\Lambda}^{-1}(C)$
\[
\xygraph{!{<0cm,0cm>;<0.6cm,0cm>:<0cm,0.9cm>::}
!{(0,-3)}="2"
!{(1,-2)}*+{2}="3"
!{(2,-1)}*+{1}="4"
!{(3,0)}*+{3}="5"
!{(4,-1)}*+{1}="6"
!{(5,-2)}*+{2}="7"
!{(6,-1)}*+{1}="8"
!{(7,-1)}="9"
!{(8.5,-1)}="10"  
!{(9.5,-2)}*+{2}="12"
!{(10.5,-1)}*+{1}="13"
!{(11.5,0)}*+{3}="14"
!{(12.5,-1)}*+{1}="15"
!{(13.5,0)}="16"
!{(14,-1)}="30"
!{(15,-1)}*+{1}="22"
!{(16,-2)}*+{2}="23"
!{(17,-1)}*+{1}="24"
!{(18,0)}*+{3}="25"
!{(19,-1)}*+{1}="26"
!{(20,-2)}*+{2}="27"
!{(21,-3)}="28"
"3":@[green]@{-->}"2"
"4":@[green]@{=>}"3"_{\alpha_2}
"5":@[green]@{=>}"4"_{\gamma_1}
"5":@[green]@{=>}"6"^{\gamma_2}
"6":"7"^{\alpha_1}
"8":"7"^{\alpha_2}
"8":@{.}"9"
"10":"12"
"13":@{=>}"12"_{\alpha_i}
"14":@{=>}"13"_{\gamma_j}
"14":@{=>}"15"^{\gamma_i}
"16":"15"
"30":@{.}"22"
"22":"23"_{\alpha_1}
"24":"23"_{\alpha_2}
"25":@[green]@{=>}"24"_{\gamma_1}
"25":@[green]@{=>}"26"^{\gamma_2}
"26":@[green]@{=>}"27"^{\alpha_1}
"27":@[green]@{-->}"28"
}
\]
Hence the vertex $1$ cannot appear in the socle of the diagram of $C$.
\end{proof}
\begin{remark}\label{five}
Let $\Lambda'=\Pqb$ be the finite-dimensional Jacobian algebra. Then it is easy to see, that its $5$-truncation $\Lambda'_5$ equals the $5$-truncation of the infinite dimensional algebra $\Lambda=\Pq$. By Lemma \ref{form} all $E_{\Lambda}$-rigid strings are $E_{\Lambda_5}$-rigid and hence $E_{\Lambda'_5}$-rigid. Thus by Prop \ref{pro} any $E_{\Lambda}$-rigid module is also $E_{\Lambda'}$-rigid. 
\end{remark}
\begin{lemma}
Let $C$ and $C'$ be given by sequences $(x_1:a_1,\ldots,a_n)$ and $(x_1:a'_1,\ldots,a'_m)$ respectively and denote by $\overline{C}$ (resp. $\overline{C'}$) the string given by $(x^-_1:a_1,\ldots,a_n)$ (resp. $(x^-_1:a'_1,\ldots,a'_m)$. Then the following hold:
\begin{compactenum}[(i)]
\item $C$ is $E_{\Lambda}$-rigid if and only if $\overline{C}$ is $E_{\Lambda}$-rigid.
\item $M(C)\oplus M(C')$ is $E_{\Lambda}$-rigid if and only if $M(\overline{C})\oplus M(\overline{C'})$ is $E_{\Lambda}$-rigid.
\end{compactenum}
\end{lemma}
\begin{proof}
The diagram of $\overline{C}$ without labelling is obtained by mirroring the diagram of $C$ horizontally. Hence there is a bijection between $\mathscr{A}(C, _c C_c)$ and $\mathscr{A}(_h \overline{C}_h,\overline{C})$. By the above remark we can consider $C$ and $\overline{C}$ as strings over $\Lambda'_5$ and use Lemma \ref{plam} to obtain
\[\left|\mathscr{A}(\tau_{\Lambda'_5}^-C,C)\right|= \left|\mathscr{A}(C,\tau_{\Lambda'_5} C)\right|= \left|\mathscr{A}(C, _c C_c)\right|=\left|\mathscr{A}(_h \overline{C}_h,\overline{C})\right|=\left|\mathscr{A}(\tau_{\Lambda'_5}^-\overline{C},\overline{C})\right| \]
and this proves (i). Now the second statement follows using the same arguments and additivity of the Auslander-Reiten translate.
\end{proof}
\begin{remark}
By the last lemma we can now restrict our investigations to strings starting with $x_1 \in Q_1$.
Because of the symmetry of $Q$ and of the relations defined in $I$, the properties of a string $C$ given by $(x_1:a_1,\ldots,a_n)$ we are interested in, do not depend on its starting arrow $x_1$. Hence we often write $(a_1,\ldots,a_n)$ instead of $(x_1:a_1,\ldots,a_n)$. If $s_1$ and $s_2$ are subsequences of $(a_1,\ldots,a_n)$ say $s_1=(a_i,a_{i+1},\ldots,a_j)$ for $1\leq i\leq j\leq n$ and $s_2=(a_k,a_{k+1},\ldots,a_l)$ for $1\leq k\leq l\leq n$ we write
\[(a_1,\ldots,a_n)=(\ldots,s_1,\ldots,s_2,\ldots).\]
This does not necessarily imply that the subsequences appear in this order in the sequence $(a_1,\ldots,a_n)$ and they may overlap. If a subsequence $s$ of $(a_1,\ldots,a_n)$ satisfies $s=a_1,\ldots,a_i$ for some $1\leq i \leq n$ we denote it by
$(a_1,\ldots,a_n)=(s,\ldots)$.
If we want to imply, that a subsequence $s$ can be a subsequence at the end of $(a_1,\ldots,a_n)$, i.e. that $s$ is not necessarily followed by a comma, we write
$(a_1,\ldots,a_n)=(\ldots ,s\ldots)$.
\end{remark}
\begin{lemma}\label{grund}
Let $C$ and $C'$ be given by $(x_1:a_1,a_2,\ldots ,a_n)$ and $(x_1:a_1',a_2',\ldots ,a_m')$ respectively with $x_1 \in Q_1$. Then $E_{\Lambda}(C',C)=0$ is equivalent to the conditions:
\begin{compactenum}[(I)]
\item Let $a:=\max \lbrace a_i \mid 1 \leq i \leq n \rbrace$. Then $ a-1\leq a_j' $ for all $1 \leq j \leq m$.
\item a)If there is a subsequence $s$ of $(a_1,a_2,\ldots ,a_n)$ such that
\[(a_1,a_2,\ldots ,a_n)=(s,a_i\ldots)\quad\text{ and }\quad(a_1',a_2',\ldots ,a_m')=(\ldots ,s,a_j'\ldots)\]
then $a_i \leq a_j'$ (in particular $a_1 \leq a_j'$ for all $1<j \leq m$).

b)If there is a subsequence $s$ of $(a_1,a_2,\ldots ,a_n)$ such that
\[(a_1,a_2,\ldots ,a_n)=(\ldots a_i,s)\quad\text{ and }\quad(a_1',a_2',\ldots ,a_m')=(\ldots a_j',s,\ldots)\]
then $a_i \leq a_j'$ (in particular $a_n \leq a_j'$ for all $1\leq j < m$).
\item If there is a subsequence $s$ of $(a_1,a_2,\ldots ,a_n)$ such that
\[(a_1,a_2,\ldots ,a_n)=(\ldots a_i,s,a_{i+t}\ldots)\text{ and } (a_1',a_2',\ldots ,a_m')=(\ldots  a_j',s,a_{j+t}'\ldots)\]
then $a_i \leq a_j'$ or $a_{i+t}\leq a_{j+t}'$.
\item If there exists $1 \leq i \leq m$ such that $(a_i',a_{i+1}',\ldots ,a_{i+n}')=(a_1,a_2,\ldots ,a_n)$ then $i=1$ or $i+n=m$.  
\end{compactenum}
\end{lemma}
\begin{proof}
%
The diagram of $\tau_{\Lambda}^{-1}(C)$ is of the form:
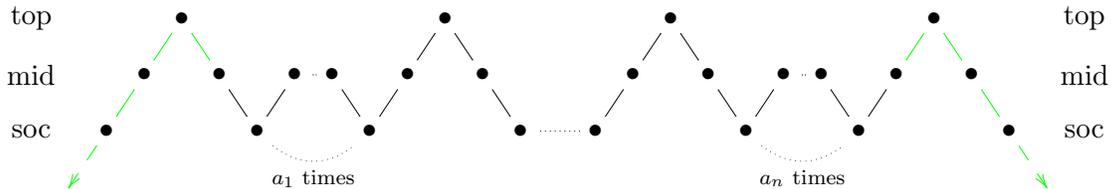
\[
\xygraph{!{<0cm,0cm>;<0.5cm,0cm>:<0cm,0.75cm>::}
!{(-5,1)}*+{\text{top}}
!{(-5,0)}*+{\text{mid}}
!{(-5,-1)}*+{\text{soc}}
!{(-4,-2)}="30"
!{(-3,-1)}*+{\bullet}="31"
!{(-2,0)}*+{\bullet}="32"
!{(-1,1)}*+{\bullet}="33"
!{(0,0)}*+{\bullet}="1"
!{(1,-1)}*+{\bullet}="2"
!{(2,0)}*+{\bullet}="3"
!{(3,0)}*+{\bullet}="4"
!{(4,-1)}*+{\bullet}="5"
!{(5,0)}*+{\bullet}="6"
!{(6,1)}*+{\bullet}="7"
!{(7,0)}*+{\bullet}="9"
!{(8,-1)}*+{\bullet}="10"
!{(10,-1)}*+{\bullet}="13"
!{(11,0)}*+{\bullet}="14"
!{(12,1)}*+{\bullet}="17"
!{(13,0)}*+{\bullet}="18"
!{(14,-1)}*+{\bullet}="19"
!{(15,0)}*+{\bullet}="20"
!{(16,0)}*+{\bullet}="21"
!{(17,-1)}*+{\bullet}="22"
!{(18,0)}*+{\bullet}="23"
!{(19,1)}*+{\bullet}="40"
!{(20,0)}*+{\bullet}="41"
!{(21,-1)}*+{\bullet}="42"
!{(23,1)}*+{\text{top}}
!{(23,0)}*+{\text{mid}}
!{(23,-1)}*+{\text{soc}}
!{(22,-2)}="43"
"33"-@[green]"32"
"32"-@[green]"31"
"31":@[green]@{-->}"30"
"33"-@[green]"1"
"1"-"2"
"2"-"3"
"2":@/_0.4cm/@{.}_{a_1 \text{ times}}"5"
"3":@{.}"4"
"4"-"5"
"5"-"6"
"6"-"7"
"7"-"9"
"10":@{.}"13"
"9"-"10"
"13"-"14"
"14"-"17"
"17"-"18"
"18"-"19"
"19"-"20"
"19":@/_0.4cm/@{.}_{a_n \text{ times}}"22"
"20":@{.}"21"
"21"-"22"
"22"-"23"
"40"-@[green]"23"
"40"-@[green]"41"
"41"-@[green]"42"
"42":@[green]@{-->}"43"
}
\]
In the diagrams of $C'$ and $\tau_{\Lambda}^{-1}(C)$ we have three levels. We call them the top, the middle and the socle. If $((C_1,C_2,C_3),(C'_1,C'_2,C'_3))$ is an admissible pair of $\tau_{\Lambda}^{-1}(C)$ and $C'$, then the middle terms $C_2$ and $C'_2$ have to start at the same level and end at the same level. Obviously $C'_2$ cannot begin or end at the top level. We consider the other possibilities case by case:
\begin{compactenum}[(i)]
\item If $C_2$ begins at socle and ends at middle, then 
\[ C_2 =    (x_2^- x_1)^b y_2 y_1^- \underbrace{x_2^- \ldots  x_1}_{=: \text{s}}y_2y_1^- x_2^-\]
and it has subdiagram in $\tau_{\Lambda}^{-1}(C)$
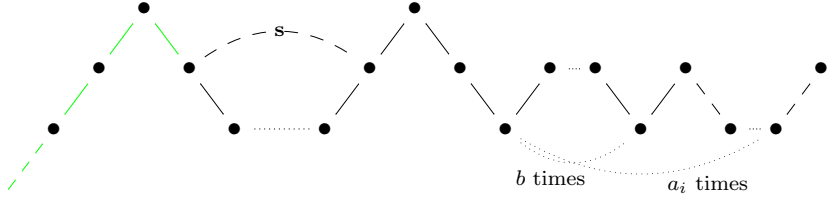
\[
\xygraph{!{<0cm,0cm>;<0.6cm,0cm>:<0cm,0.8cm>::}
!{(-1,-1)}="0"
!{(0,0)}*+{\bullet}="1"
!{(1,1)}*+{\bullet}="2"
!{(2,2)}*+{\bullet}="3"
!{(3,1)}*+{\bullet}="4"
!{(4,0)}*+{\bullet}="5"
!{(6,0)}*+{\bullet}="6"
!{(7,1)}*+{\bullet}="7"
!{(8,2)}*+{\bullet}="8"
!{(9,1)}*+{\bullet}="9"
!{(10,0)}*+{\bullet}="10"
!{(11,1)}*+{\bullet}="11"
!{(12,1)}*+{\bullet}="12"
!{(13,0)}*+{\bullet}="13"
!{(14,1)}*+{\bullet}="14"
!{(15,0)}*+{\bullet}="15"
!{(16,0)}*+{\bullet}="16"
!{(17,1)}*+{\bullet}="17"
"1":@[green]@{--}"0"
"1"-@[green]"2"
"3"-@[green]"2"
"3"-@[green]"4"
"4"-"5"
"5":@{.}"6"
"4":@/^0.5cm/@{--}"7"|-{\textbf{s}}
"6"-"7"
"7"-"8"
"8"-"9"
"9"-"10"
"10"-"11"
"10":@/_0.45cm/@{.}_(0.4){b \text{ times}}"13"
"10":@/_0.6cm/@{.}_(0.7){a_i \text{ times}}"16"
"11":@{.}"12"
"12"-"13"
"13"-"14"
"14":@{--}"15"
"15":@{.}"16"
"16":@{--}"17"
}
\]
where $b < a_i$ and the broken lines are not part of the diagram of $C_2$ but of the diagram of $\tau_{\Lambda}^{-1}(C)$. Hence in $C'$, $C'_2$ corresponds to a subsequence $(\ldots ,s,b,\ldots)$. This is eliminated by condition (II)(a). On the other hand if $E_{\Lambda}(C',C)=0$, we see that $C$ and $C'$ satisfy (II)(a).
\item If $C_2$ begins at middle and ends at socle we use the same argument as above and condition (II)(b). And if $E_{\Lambda}(C',C)=0$ holds, we see that $C$ and $C'$ satisfy condition (II)(b).
\item If $C_2$ begins and ends at middle then either
$ C_2 = (x_2^- x_1)^b$ with $b<a-1$ and with subdiagram in $\tau_{\Lambda}^{-1}(C)$ given by:
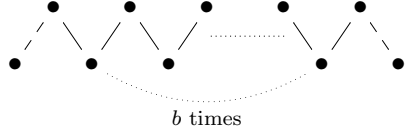
\[
\xygraph{!{<0cm,0cm>;<0.51cm,0cm>:<0cm,0.75cm>::}
!{(-1,-1)}*+{\bullet}="0"
!{(0,0)}*+{\bullet}="1"
!{(1,-1)}*+{\bullet}="2"
!{(2,0)}*+{\bullet}="3"
!{(3,-1)}*+{\bullet}="7"
!{(4,0)}*+{\bullet}="4"
!{(6,0)}*+{\bullet}="5"
!{(7,-1)}*+{\bullet}="8"
!{(8,0)}*+{\bullet}="6"
!{(9,-1)}*+{\bullet}="11"
!{(4,-0.5)}="9"
!{(6,-0.5)}="10"
"1":@{--}"0"
"1"-"2"
"2"-"3"
"7"-"3"
"7"-"4"
"2":@/_0.5cm/@{.}_{b \text{ times}}"8"
"5"-"8"
"8"-"6"
"9":@{.}"10"
"6":@{--}"11"
}
\]
Then $C'_2$ corresponds to a subsequence $(\ldots ,b,\ldots)$ in $C'$. Since $b < a-1$ this is ruled out by condition (I). On the other hand if $a_j'=b < a-1$ for one $1\leq j \leq m$ we see that there is a non-trivial homomorphism from $\tau_{\Lambda}^{-1}(C)$ to $C'$. Hence if $E_{\Lambda}(C',C)=0$, condition (I) is satisfied. Otherwise $ C_2=   (x_2^-x_1 )^{b_2}y_2 y_1^-\; s \;y_2y_1^-(x_2^-x_1)^{b_1}$ where $(a_1,\ldots,a_n)=(\ldots a_i,s,a_{i+t} \ldots)$ 
has subdiagram
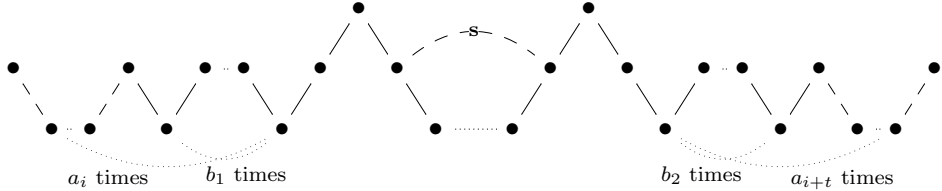
\[
\xygraph{!{<0cm,0cm>;<0.51cm,0cm>:<0cm,0.8cm>::}
!{(-7,1)}*+{\bullet}="32"
!{(-6,0)}*+{\bullet}="33"
!{(-5,0)}*+{\bullet}="30"
!{(-4,1)}*+{\bullet}="29"
!{(-3,0)}*+{\bullet}="210"
!{(-2,1)}*+{\bullet}="211"
!{(-1,1)}*+{\bullet}="212"
!{(0,0)}*+{\bullet}="1"
!{(1,1)}*+{\bullet}="2"
!{(2,2)}*+{\bullet}="3"
!{(3,1)}*+{\bullet}="4"
!{(4,0)}*+{\bullet}="5"
!{(6,0)}*+{\bullet}="6"
!{(7,1)}*+{\bullet}="7"
!{(8,2)}*+{\bullet}="8"
!{(9,1)}*+{\bullet}="9"
!{(10,0)}*+{\bullet}="10"
!{(11,1)}*+{\bullet}="11"
!{(12,1)}*+{\bullet}="12"
!{(13,0)}*+{\bullet}="13"
!{(14,1)}*+{\bullet}="14"
!{(15,0)}*+{\bullet}="31"
!{(16,0)}*+{\bullet}="40"
!{(17,1)}*+{\bullet}="41"
"32":@{--}"33"
"30":@{.}"33"
"29":@{--}"30"
"29"-"210"
"210"-"211"
"33":@/_0.5cm/@{.}_(0.3){a_i \text{ times}}"1"
"210":@/_0.4cm/@{.}_(0.6){b_1 \text{ times}}"1"
"211":@{.}"212"
"212"-"1"
"1"-"2"
"3"-"2"
"3"-"4"
"4"-"5"
"5":@{.}"6"
"4":@/^0.5cm/@{--}"7"|-{\textbf{s}}
"6"-"7"
"7"-"8"
"8"-"9"
"9"-"10"
"10"-"11"
"10":@/_0.4cm/@{.}_(0.4){b_2 \text{ times}}"13"
"10":@/_0.5cm/@{.}_(0.7){a_{i+t} \text{ times}}"40"
"11":@{.}"12"
"12"-"13"
"13"-"14"
"14":@{--}"31"
"31":@{.}"40"
"40":@{--}"41"
}
\]
Then $C'_2$ corresponds to a subsequence $(\ldots ,b_1,s,b_2,\ldots)$ in $C'$. Since $b_1 < a_i$ and $b_2 < a_{i+t}$ this case is excluded by condition (III). Again assuming $E_{\Lambda}(C',C)=0$ this shows that $C$ and $C'$ satisfy condition (III)
\item If $C_2$ starts and ends at socle, then $C'_2$ corresponds to $( ,a_1,a_2,\ldots,a_n,)$. This is no subsequence of $C'$ by condition (IV). On the other hand if it was a subsequence then $E_{\Lambda}(C',C)=0$ could not be true.
\end{compactenum}
\end{proof}
Let $C$ be a string given by the sequence $(x_1:a_1,a_2,\ldots ,a_n)$. We set
\[(x_1:a_1,a_2,\ldots ,a_n)^{\sigma}:=(x_1:a_n,\ldots ,a_2,a_1)\]
and say $C$ is \textit{symmetric} if $(x_1:a_1,a_2,\ldots ,a_n)^{\sigma}=(x_1:a_1,a_2\ldots ,a_n)$.  
\begin{prop}\label{rig}
A string $C$ of the form $(x_1:a_1,a_2,\ldots ,a_n)$ is $E_{\Lambda}$-rigid if and only if the following hold:
\begin{compactenum}[(i)]
\item $C$ is symmetric.
\item Let $a:=a_1$. Then $a_i \in \lbrace a, a+1\rbrace$ for all $1\leq i\leq n$.
\item There is no subsequence $s$ of $(a_1,a_2,\ldots ,a_n)$ such that
\[(a_1,a_2,\ldots ,a_n)=(s,a+1\ldots ,s,a\ldots).\]
\item There is no subsequence $s$ of $(a_1,a_2,\ldots ,a_n)$ such that
\[(a_1,a_2,\ldots ,a_n)=(\ldots a+1,s,a+1\ldots a,s,a\ldots).\]
\end{compactenum}
\end{prop}
\begin{proof}
Let $C=(x_1:a_1,a_2,\ldots ,a_n)$ satisfy conditions $(i)-(iv)$. We use  Lemma \ref{grund} to show that
$C$ is $E_{\Lambda}$-rigid. Since $a_i \in \lbrace a, a+1\rbrace$ for all $1\leq i\leq n$ condition (I) holds. The symmetry of $C$ and (iii) imply Lemma \ref{grund} (II). Condition (iv) guarantees Lemma \ref{grund} (III). Lemma \ref{grund} (IV) is obviously satisfied.

Now assume that $C$ is $E_{\Lambda}$-rigid. Then $C$ satisfies the conditions on $C$ and $C'=C$ in Lemma \ref{grund}. Let $b:=\max \lbrace a_i \mid 1\leq i\leq n\rbrace $. Then by condition Lemma \ref{grund}(I) we have $b-1 \leq a_i \leq b$ for all $1 \leq i \leq n$. By Lemma \ref{grund} (II) we have  $a_1\leq a_i$ for all $1 \leq i \leq n$ and this shows (ii). Conditions (iii) and (iv) follow directly from Lemma \ref{grund} (II) and (III). 

Finally we prove that $C$ is symmetric, that is 
$(a_1,a_2,\ldots,a_n)=(a_n,\ldots,a_2,a_1)$.
Assume that this is not the case and choose a subsequence $s$ maximal such that
\[C=(\underbrace {a_1,a_2,\ldots,a_k}_s,a_{k+1},\ldots,a_{n-k+1},\underbrace{a_k,\ldots,a_2,a_1}_{s^{\sigma}})\]
with $a_{k+1} \neq a_{n-k+1}$. We already know that $C$ satisfies (ii), hence $\lbrace a_{k+1},a_{n-k+1}\rbrace=\lbrace a,a+1\rbrace$, where $a:=a_1$. Without loss of generality assume $a_{n-k+1}=a=a_1$ and $a_{k+1}=a+1$. 
Define $s'$ to be the subsequence of $C$ of same length as $s$ starting at $a_{n-k+1}$, hence $s'=(a,a_k,\ldots,a_2)$. Since we have already shown that $C$ satisfies condition (iii) we know that $s \neq s'$. Choose $P_1$ maximal such that $s=(P_1,P_2)$ and $s'=(P_1,P'_2)$. Since $C$ satisfies (iii) we know that $P_2=(a,\ldots)$ and $P'_2=(a+1,\ldots)$ and hence
\[ C=(\underbrace{\underbrace{a,Q}_{P_1},\underbrace{a,\ldots}_{P_2}}_s ,a+1,\ldots,\underbrace{\underbrace{a,Q}_{P_1},\underbrace{a+1,\ldots}_{P'_2}}_{s'},a)\]
for some subsequence $Q$. In particular we know that 
\[  s^{\sigma}=(Q,a+1,\ldots,a, Q^{\sigma},a)\quad \text{and}\quad s=(s^{\sigma})^{\sigma}=(a,Q,a,\ldots,a+1, Q^{\sigma}).\] 
So eventually we have 
\[ C=(s,a+1,\ldots,a, s^{\sigma})=(a,Q,a ,\ldots,  \boldsymbol{a+1 , Q^{\sigma} ,a+1},\ldots ,a,Q,a+1,\ldots,\boldsymbol{a,Q^{\sigma},a}) \] 
which contradicts condition (iv).
\end{proof}
\begin{remark}
It is easy to see, slightly altering the proofs of Lemma \ref{grund} and Proposition \ref{rig}, that any $E_{\Lambda}$-rigid module has trivial endomorphism ring. However, the module corresponding to the string $\beta_2 \alpha_1$ has trivial endomorphism ring and no selfextensions but is not $E_{\Lambda}$-rigid.
\end{remark}
\begin{lemma}
If 
\[C=(x_1:a_1,\ldots,a_n) \qquad \text{and} \qquad C'=(x'_1:a'_1,\ldots,a'_m)\] are $E_{\Lambda}$-rigid strings, not of length zero and such that $x_1\neq x'_1$, then there is no edge between them in $\Gamma(\srdirr (\Lambda))$. 
\end{lemma}
\begin{proof}
This can be seen by fixing $x_1$ and considering all five possible cases for $x'_1 \neq x_1$. 
\end{proof}
\begin{corollary}\label{edge}
Let $C$ and $C'$ in be $E_{\Lambda}$-rigid given by a sequence $(x_1:a_1,a_2,\ldots ,a_n)$ and $(x_1:a_1',a_2',\ldots ,a_m')$ respectively. Set $a:=a_1$ and assume $a_1'\leq a_1$.  Then 
\begin{equation*}
E_{\Lambda}(C',C)=0=E_{\Lambda}(C,C')
\end{equation*}
if and only if the following hold:
\begin{compactenum}[(i)]
\item We have $a_1'\in \lbrace a-1,a\rbrace$ and if $a_1'=a-1$ then $m=1$ and $a_i=a$ for all $1 \leq i \leq n$.
\item a)If there is a subsequence $s$ of $(a_1,a_2,\ldots ,a_n)$ such that
\[(a_1,a_2,\ldots ,a_n)=(s,a+1\ldots)\text{ and }(a_1',a_2',\ldots ,a_m')=(\ldots ,s,a_j'\ldots)\]
then $a_j' = a+1$.

b)The same statement as in a) holds with reversed roles of $C$ and $C'$.
\item a)If there is a subsequence $s$ of $(a_1,a_2,\ldots ,a_n)$ such that
\[(a_1,a_2,\ldots ,a_n)=(\ldots,a+1,s,a+1,\ldots)\] and \[(a_1',a_2',\ldots ,a_m')=(\ldots a_j',s,a_{j+t}'\ldots)\]
then $a_j'= a+1$ or $a_{j+t}'= a+1$ (here $t=1$ is possible). 

b)The same statement as in a) holds with reversed roles of $C$ and $C'$.
\item a)If there exists $1 \leq i \leq m$ such that $(a_i',a_{i+1}',\ldots ,a_{i+n}')=(a_1,a_2,\ldots ,a_n)$ then $i=1$ or $i+n=m$.

b)The same statement as in a) holds with reversed roles of $C$ and $C'$.
\end{compactenum}
\end{corollary}
\begin{proof}
Use Lemma \ref{grund} and Proposition \ref{rig}.
\end{proof}

\subsection{The components that are not $E_{\Lambda}$-rigid}
In this subsection we describe the strongly reduced bands. By a sequence $(x_1:a_1,a_2,\ldots ,a_n,)$, where $x_1 \in \lbrace \alpha_1^{\pm}, \beta_1^{\pm}, \gamma_1^{\pm}\rbrace $ and $a_i \in \mathbb{N}_{\geq 1}$ ending with comma we denote the string
\[B= y_2y_1^- (x_2^- x_1)^{a_n} y_2 y_1^-(x_2^- x_1)^{a_{n-1}} \ldots (x_2^- x_1)^{a_2}y_2y_1^-(x_2^- x_1)^{a_1}\]
where $y_1\in \lbrace \alpha_1^{\pm}, \beta_1^{\pm}, \gamma_1^{\pm}\rbrace$. 

If $B$ corresponds to the sequence $(x_1:a_1,a_2,\ldots,a_n,)$ then $s(B)=t(B)$ and for any positive integer $k$ the string $B^k$ is given by the sequence
\begin{align*}
(x_1:a_1,\ldots,a_n,)^k
=(x_1:a_1,a_2,\ldots,a_n,a_1,\ldots,a_n,\ldots a_1,\ldots,a_n,).
\end{align*}
\begin{remark}
Let $B$ be given by a sequence $(x_1:a_1,a_2,\ldots ,a_n,)$. Even though we always have $s(B)=t(B)$ this $B$ is not a band in general. Consider for example the sequence $(\alpha_1:2,2,)$. This defines the string
$B=\gamma_2\gamma_1^-\alpha_2^-\alpha_1\alpha_2^-\alpha_1  \gamma_2 \gamma_1^-\alpha_2^-\alpha_1\alpha_2^-\alpha_1  =(\gamma_2\gamma_1^-\alpha_2^-\alpha_1\alpha_2^-\alpha_1  )^2$ which is not a band as it is the power of a shorter string.
Also note that if $(x_1:a_1,\ldots,a_n,)$ describes a band, then for any $1\leq i \leq n$ the rotation $(x_1:a_i,\ldots,a_n,a_1,\ldots ,a_{i-1},)$ describes a rotation-equivalent band. However, we will establish conditions on the positive integers $a_i$ such that the sequences $(x_1:a_1,a_2,\ldots,a_n,)$ satisfying these conditions describe a complete set of representatives of strongly reduced bands which are pairwise not rotation-equivalent. 
\end{remark}
\begin{lemma}\label{formband}
Up to rotation-equivalence any strongly reduced band $B$ is of the form $(x_1:a_1,\ldots ,a_n,)$.
\end{lemma}
\begin{proof}
The proof that any strongly reduced band $B$ is of the form $(x_1:a_1,a_2,\ldots ,a_n,)$ is similar to the proof of Lemma \ref{form}. Since $\tau_{\Lambda}^{-1}(B)$ equals $B$, it is still not quite the same and we have to check the details again. So let $B=c_m\ldots c_1$ be strongly reduced and assume without loss of generality that $c_1 \in Q_1$. 

First note the following: If $B$ is strongly reduced, there cannot be a directed substring of length $\geq 6$. A directed substring of length $6$ would mean, that there was the same vertex at the top and the socle, yielding a non-trivial homomorphism. A directed substring of length $>6$ means that we walked through the cycle of length $6$ in $\Lambda$ more than once. Hence there is the same sequence at the top of the directed string and in the socle. 

Now we show that in fact the longest directed substring $c_{j+i}\ldots c_{j+1}$ of $B$ such that $c_{j+i}\ldots c_{j+1}\in Q_1$ is of length at most two. We assume that $c_{j+1}=\alpha_1$ and thus $c_j=\alpha_2^-$ (where $c_0=c_m$ if $j+1=1$).

Case 1: Let $i = 3$. 
Then the vertex $1$ appears in its top and in its socle which yields a non-trivial endomorphism of $B$. 


Case 2: Let $i = 4 $. Since $c_j=\alpha_2^-$ we know that $c_{j-1}\notin \lbrace \alpha_1^-,\alpha_2 \rbrace$ where $c_{j-1}=c_m$ if $j=1$. Thus the band $B$ has diagram
\[
\xygraph{!{<0cm,0cm>;<0.7cm,0cm>:<0cm,0.7cm>::}
!{(-1,1)}*+{1}="1"
!{(0,0)}*+{2}="2"
!{(1,-1)}*+{3}="23"
!{(2,-2)}*+{1}="3"
!{(3,-3)}*+{2}="4"
!{(4,-2)}*+{1}="5"
!{(6,-2)}*+{1}="19"
!{(7,-3)}*+{2}="7"
!{(8,-2)}*+{1}="8"
!{(9,-1)}*+{3}="9"
!{(10,-1)}="12"
!{(13,-1)}*+{2}="20"
"1":@{=>}"2"_{\alpha_1}
"2":"23"_{\beta_2}
"23":"3"_{\gamma_1}
"3":"4"_{\alpha_2}
"5":"4"_{\alpha_1}
"5":@{.}"19"
"19":"7"_{\alpha_2}
"8":@{=>}"7"_{\alpha_1}
"9":"8"_{\gamma_2}
"12":@{.}"20"
"1":@/^0.5cm/"20"^(0.8){\alpha_2}
}
\]
or it has diagram
\[
\xygraph{!{<0cm,0cm>;<0.7cm,0cm>:<0cm,0.7cm>::}
!{(-1,1)}*+{1}="1"
!{(0,0)}*+{2}="2"
!{(1,-1)}*+{3}="23"
!{(2,-2)}*+{1}="3"
!{(3,-3)}*+{2}="4"
!{(4,-2)}*+{1}="5"
!{(5,-3)}*+{2}="6"
!{(6,-3)}*+{2}="7"
!{(7,-2)}*+{1}="8"
!{(8,-3)}*+{2}="10"
!{(9,-4)}*+{3}="11"
!{(10,-3)}="12"
!{(13,-3)}*+{2}="20"
"1":"2"_{\alpha_1}
"2":"23"_{\beta_2}
"23":"3"_{\gamma_1}
"3":@{=>}"4"_{\alpha_2}
"5":"4"_{\alpha_1}
"5":"6"^{\alpha_2}
"6":@{.}"7"
"8":"7"_{\alpha_1}
"8":@{=>}"10"^{\alpha_2}
"10":"11"^{\beta_1}
"12":@{.}"20"
"1":@/^1cm/"20"^{\alpha_2}
}
\]
In either case we see that there is a non-trivial endomorphism.

Case 3: Let $i = 5 $. Since $c_j=\alpha_2^-$ we know that $c_{j-1}\notin \lbrace \alpha_1^-,\alpha_2, \beta_{2}^-,\beta_1\rbrace$ where $c_{j-1}=c_m$ if $j=1$. Thus the band $B$ either has diagram
\[
\xygraph{!{<0cm,0cm>;<0.7cm,0cm>:<0cm,0.7cm>::}
!{(0,0)}*+{1}="1"
!{(1,-1)}*+{2}="2"
!{(2,-2)}*+{3}="3"
!{(3,-3)}*+{1}="4"
!{(4,-4)}*+{2}="14"
!{(5,-5)}*+{3}="15"
!{(6,-4)}*+{2}="5"
!{(7,-4)}*+{2}="16"
!{(8,-5)}*+{3}="17"
!{(9,-4)}*+{2}="18"
!{(10,-3)}*+{1}="19"
!{(11,-3)}*+{1}="20"
!{(12,-4)}*+{2}="21"
!{(13,-3)}*+{1}="22"
!{(14,-2)}*+{3}="23"
!{(15,-4)}="12"
!{(17,-4)}*+{2}="30"
"1":@{=>}"2"_{\alpha_1}
"2":"3"_{\beta_2}
"3":"4"_{\gamma_1}
"4":"14"_{\alpha_2}
"14":"15"_{\beta_1}
"5":"15"^{\beta_2}
"5":@{.}"16"
"16":"17"_{\beta_1}
"18":"17"^{\beta_2}
"19":"18"^{\alpha_1}
"19":@{.}"20"
"20":"21"_{\alpha_2}
"22":@{=>}"21"^{\alpha_1}
"23":"22"^{\gamma_2}
"12":@{.}"30"
"1":@`{"1"+(15,-0.5),"30"+(-1,3)}"30"^{\alpha_2}
}
\]
or it has diagram
\[
\xygraph{!{<0cm,0cm>;<0.7cm,0cm>:<0cm,0.7cm>::}
!{(0,0)}*+{1}="1"
!{(1,-1)}*+{2}="2"
!{(2,-2)}*+{3}="3"
!{(3,-3)}*+{1}="4"
!{(4,-4)}*+{2}="14"
!{(5,-5)}*+{3}="15"
!{(6,-4)}*+{2}="5"
!{(7,-5)}*+{3}="25"
!{(8,-5)}*+{3}="26"
!{(9,-4)}*+{2}="16"
!{(10,-5)}*+{3}="17"
!{(11,-6)}*+{1}="18"
!{(12,-4)}="12"
!{(15,-4)}*+{2}="30"
"1":"2"_{\alpha_1}
"2":"3"_{\beta_2}
"3":"4"_{\gamma_1}
"4":"14"_{\alpha_2}
"14":@{=>}"15"_{\beta_1}
"5":"15"^{\beta_2}
"5":"25"^{\beta_1}
"25":@{.}"26"
"16":"26"_{\beta_2}
"16":@{=>}"17"_{\beta_1}
"17":"18"_{\gamma_2}
"12":@{.}"30"
"1":@/^1cm/"30"^{\alpha_2}
}
\]
In both cases we find a non-trivial endomorphism. So we actually have $1 \leq i \leq 2$. Note if $i$ equals $1$ then $B$ is given by the sequence $(x_1:0,)$.
 
So now let $B$ be a band with a subword $c_{i+4}c_{i+3}c_{i+2}c_{i+1}c_i$ such that $c_i ,c_{i+3} \in Q_1^-$ and $c_{i+1},c_{i+2},c_{i+4} \in Q_1$ (or vice versa). We assume that $c_{i+1}=\alpha_1$. Then $B$ has diagram
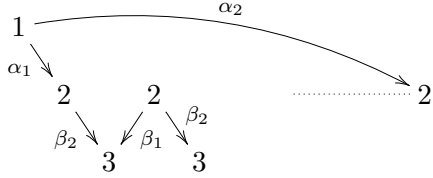
\[
\xygraph{!{<0cm,0cm>;<0.6cm,0cm>:<0cm,0.9cm>::}
!{(0,0)}*+{1}="1"
!{(1,-1)}*+{2}="2"
!{(2,-2)}*+{3}="3"
!{(3,-1)}*+{2}="4"
!{(4,-2)}*+{3}="14"
!{(6,-1)}="12"
!{(9,-1)}*+{2}="30"
"1":"2"_{\alpha_1}
"2":"3"_{\beta_2}
"4":"3"^{\beta_1}
"4":"14"^{\beta_2}
"12":@{.}"30"
"1":@/^0.5cm/"30"^{\alpha_2}
}
\]
and the vertices $1$ and $2$ cannot appear in the socle and the vertex $3$ cannot appear in the top of the diagram of $B$. If $c_{i+1}=\alpha_1^-$ 
 the vertices $1$ and $2$ cannot appear in the top and the vertex $3$ cannot appear in the socle of the diagram of $B$. In either case we see that $B$ can be described by a sequence $(x_1:a_1,\ldots,a_n,)$.
\end{proof} 
\begin{lemma}\label{formband}
Let $B=(x_1:a_1,\ldots,a_n,)$ be a band and let $a:=\min_{1 \leq i\leq n}\lbrace a_i \rbrace$. Then $B$ is strongly reduced if and only if the following hold:
\begin{compactenum}[(i)]
\item $a_i \in \lbrace a,a+1 \rbrace$ for all $1\leq i\leq n$.
\item for any positive integer $k$ there is no subsequence $s$ of $B^k=(a_1,\ldots,a_n,)^k$  such that \[(a_1,\ldots,a_n,)^k=(\ldots,a,s,a,\ldots ,a+1,s,a+1,\ldots).\] 
\end{compactenum}
\end{lemma}
\begin{proof}
This is proved in a similar way as Lemma $\ref{grund}$. 
\end{proof}
\begin{remark}
When using Lemma \ref{formband} to show that a band $B$ is strongly reduced, it will be enough to check condition (ii) for $k=2$.
\end{remark}
\begin{lemma}\label{stba}
\begin{compactenum}[(i)]
\item If the string $ C=(x_1:a_1,\ldots ,a_n)$ is $E_{\Lambda}$-rigid, then $(x_1:a_1,\ldots,a_n+1,)$ is a strongly reduced band $B$.
\item 
Let $C=(x_1:a_1,\ldots,a_n)$ and $C'=(x'_1:a'_1,\ldots,a'_m)$ be two different $E_{\Lambda}$-rigid strings. Then the sequences $(x_1:a_1,\ldots,a_n+1,)$ and $(x'_1:a'_1,\ldots,a'_m+1,)$ describe bands $B$ and $B'$, that are not rotation-equivalent.
\end{compactenum}
\end{lemma}
\begin{proof}
We first show that $B$ is indeed a band. By definition we have $s(B)=t(B)$. Now suppose that $B=X^m$ for some band $X$ and $m\geq 2$. Then we can assume that $X=(x_1:a_1,\ldots,a_k,)$ for some $1\leq k < n$ and $(a_1,\ldots,a_{n}+1,)=(a_1,\ldots,a_k,)^m$. Since $C$ is $E_{\Lambda}$-rigid, if $a:=a_1$ we know that $C=(a,a_2,\ldots,a)$ and hence $B=(a,a_2,\ldots,a+1,)$. Therefore we can assume that $X=(s,a+1,)$ for some subsequence $s$. Then 
\[ C=(\underbrace{s,a+1,s,a+1,\ldots,s,a+1}_{m-1},s,a). \]
But this contradicts $C$ being $E_{\Lambda}$-rigid.

Now we show that $B$ is strongly reduced. Part (i) of Lemma \ref{formband} is obviously satisfied. Consider the more complicated case, where we have $Q=a+1,s,a+1$ and $P=a,s,a$ as subsequences of $B^2$. We have to consider four cases. We denote two copies of $B$ as follows
\[B^2= |\underbrace{a,a_2,\ldots,a_{n-1},a+1,}_B|\underbrace{a,a_2,\ldots,a_{n-1},a+1,}_B| \]
Case 1: $Q$ does not contain $a_n+1$
\begin{compactenum}[(i)]
\item  and $P$ does not contain $a_n+1$. Then \[B=|a,\ldots \underbrace{,a+1,s,a+1,}_Q \ldots \underbrace{,a,s,a,}_P \ldots ,a+1,| \] and then $C=(a,\ldots ,\boldsymbol{a+1,s,a+1}, \ldots \boldsymbol{,a,s,a,} \ldots ,a)$
is not $E_{\Lambda}$-rigid.  
\item and $P$ does contain $a_n+1$. Then\[B^2=|a,\ldots ,a+1,s,a+1, \ldots ,a,\underbrace{s_1,a+1,|s_2}_s ,a, \ldots ,a+1,| \]
and then $ C = (s_2 ,a,\ldots , \boldsymbol{ a+1,s_1,a+1} ,s_2 ,a+1 \ldots \boldsymbol{ a,s_1,a })$ 
is not $E_{\Lambda}$-rigid.
\end{compactenum}
Case 2: $Q =a+1,s,a_n+1$
\begin{compactenum}[(i)]
\parindent-5mm
\item and $P$ does not contain $a_n+1$. Then $B=|a,\ldots ,a,s,a, \ldots ,a+1,s,a+1,|$
and then $C=(a,\ldots ,a,s,a, \ldots,a+1,s,a)$
is not $E_{\Lambda}$-rigid.
\item and $P$ does contain $a_n+1$. Then \[B^2=|a,\ldots ,a,\underbrace{s_1,a+1,|s_2}_s,a, \ldots ,a+1,s,a+1,| \]
and then $C=(s_2,a,\ldots ,\boldsymbol{ a+1,s_1,a+1},s_2,a)=(s_2,a,\ldots ,\boldsymbol{ a,s_1,a})$
is not $E_{\Lambda}$-rigid.
\end{compactenum}
Case 3: $Q =a_n+1,s,a+1$
\begin{compactenum}[(i)]
\item and $P$ does not contain $a_n+1$. Then \[B^2=|a,\ldots ,a,s,a, \ldots ,a+1,|s,a+1,\ldots| \]
and then $C=(\boldsymbol{ s,a+1},\ldots ,a,\boldsymbol{ s,a,} \ldots,a)$
is not $E_{\Lambda}$-rigid.
\item and $P$ does contain $a_n+1$. Then \[B^3=|\ldots,a+1,|s,a+1,\ldots ,a,\underbrace{s_1,a+1,|s_2}_s,a, \ldots ,a+1,| \]
and then $C=(\boldsymbol{ s_1,a+1},s_2,a+1,\ldots ,a,\boldsymbol{ s_1,a})$
is not $E_{\Lambda}$-rigid.
\end{compactenum}
Case 4: $Q=a+1,s_1,a_n+1,s_2,a+1$ 
\begin{compactenum}[(i)]
\item and $P$ does not contain $a_n+1$. Then \[B^2=|a,\ldots ,a,s,a, \ldots ,a+1,\underbrace{s_1,a+1,|s_2}_s ,a+1,\ldots,a+1,| \]
and then $C=(\boldsymbol{ s_2,a+1},\ldots ,a,s_1,a+1,\boldsymbol{ s_2,a}, \ldots,a+1,s_1,a)$
is not $E_{\Lambda}$-rigid.
\item and $P$ does contain $a_n+1$. Then \[B^3=|\ldots ,a,\underbrace{\overline{s_1},a+1,|\overline{s_2}}_s,a,\ldots ,a+1,\underbrace{s_1,a+1,|s_2}_s,a+1, \ldots ,a+1,| \]
and if $\left|s_2\right|< \left|\overline{s_2}\right|$ we have $C=(\overline{s_2},a,\ldots)=(s_2,a+1,\ldots ,s_2,a, \ldots)$ which
is not $E_{\Lambda}$-rigid.
If $|s_1|<|\overline{s_1}|$ we have $C=(\ldots,a,\overline{s_1},a)=(\ldots, a,s_1,a,\ldots,a+1,s_1,a)$ which is not $E_{\Lambda}$-rigid (here $|s|$ denotes the length of the sequence $s$).
\end{compactenum}

To show part (ii) of the lemma, assume that $B$ and $B'$ are rotation-equivalent. We have to show, that $C=C'$. We see immediately that $m=n$. If $B$ is given by $(x_1:1,)$ then any rotation-equivalent band $B'$ is given by $(x'_1:1,)$. Thus $C$ is associated to the sequence $(x_1:0)$ and $C'$ to the sequence $(x'_1:0)$. Then by definition we have $x_1,x'_1 \in \lbrace \alpha_1, \beta_1, \gamma_1\rbrace$ and thus we must have $x_1=x'_1$. Now assume that $a_n\geq 1$ and $a'_n\geq 1$. Then in particular we have $a_n+1\geq 2$ and $a'_n+1\geq 2$ and we can compare the vertices appearing in the top, middle and socle of the diagrams of $B$ and $B'$ to see that we must have $x_1=x'_1$. Then the sequences $(x_1:a_1,\ldots,a_{n-1},a_n+1)$ and $(x_1:a'_1,\ldots,a'_{n-1},a'_n+1)$ describe rotation-equivalent bands if and only if there exists some $2 \leq i \leq n$ such that $(a_1,\ldots,a_n+1)=(a'_i,\ldots ,a'_n+1,a'_1,\ldots ,a_{i-1})$. Then, by using the description of the corresponding $E_{\Lambda}$-rigid strings, it easily follows that $a:=a_1=a'_1$. Now let $1\leq i \leq n$ be maximal with $a_j=a'_j$ for all $1\leq j \leq i$ and assume without loss of generality that $a_{i+1}=a$ and $a'_{i+1}=a+1$. Then because of the symmetry of the strings $C$ and $C'$ we have
\[C=(a_1,\ldots, a_i, a,\ldots ,a,a_i,\ldots,a_1)\] and 
\[C'=(a_1,\ldots, a_i, a+1,\ldots ,a+1,a_i,\ldots,a_1).\] But then we have
\[B^2=|a_1,a_2,\ldots,a_i,a,\ldots,a,\underbrace{a_i,\ldots,a_2,a_1 +1|a_1,a_2,\ldots,a_i}_{=:s},a,\ldots,a_1 +1| \]
and \[B'^2=|a_1,a_2,\ldots,a_i,a+1,\ldots,a+1,\underbrace{a_i,\ldots,a_2,a_1 +1|a_1,a_2,\ldots,a_i}_{=s},a+1,\ldots,a_1 +1| \]
contradicting that $B$ is strongly reduced if $B$ and $B'$ are rotation-equivalent. 
\end{proof}
The next corollary is proven in a similar way as Corollary \ref{edge}
\begin{corollary}\label{edgeband}
Let $C$ be $E_{\Lambda}$-rigid
and $B$ strongly reduced given by $(x_1:a_1,a_2,\ldots ,a_n)$ and $(x_1:a_1',a_2',\ldots ,a_m',)$ respectively. Set $a:=a_1$ and assume $a=\min_{1 \leq i\leq n}\lbrace a_i \rbrace$. Then 
\begin{equation*}
E_{\Lambda}(B,C)=0=E_{\Lambda}(C,B)
\end{equation*}
if and only if the following hold:
\begin{compactenum}[(i)]
\item If there is a subsequence $s$ of $(a_1,a_2,\ldots ,a_n)$ and some positive integer $k$ such that
\[(a_1,a_2,\ldots ,a_n)=(s,a+1\ldots)\text{ and }(a_1',a_2',\ldots ,a_m')^k=(\ldots ,s,a_j'\ldots)\]
then $a_j' = a+1$.
\item a)If there is a subsequence $s$ of $(a_1,a_2,\ldots ,a_n)$ such that
\[(a_1,a_2,\ldots ,a_n)=(\ldots,a+1,s,a+1,\ldots)\] and some positive integer $k$ such that\[(a_1',a_2',\ldots ,a_m')^k=(\ldots a_j',s,a_{j+t}'\ldots)\]
then $a_j'= a+1$ or $a_{j+t}'= a+1$ (here $t=1$ is possible). 

b)The same statement as in a) holds with reversed roles of $C$ and $B$.
\item The sequence \[s=,a_1,a_2,\ldots ,a_n,\] (note the comma at the beginning and end of $s$) is no subsequence of $(a_1',a_2',\ldots ,a_m')^k$ for any positive integer $k$.
\end{compactenum}
\end{corollary}
\section{The graph of strongly reduced components is connected}
\subsection{Notation}
In this section we introduce a slightly more complicated notation for the strongly reduced components. This notation allows us to describe the neighbours of a given vertex of $\Gamma$ explicitly. We describe the strongly reduced components by sequences $(x_1:a|k_0|k_1| \ldots |k_m)$, where $x_1 \in \lbrace \alpha_1^{\pm}, \beta_1^{\pm}, \gamma_1^{\pm}\rbrace $, and $a,m$ and $k_0, \ldots, k_m$ are natural numbers. More precisely let
\begin{align*}
\Psi = \lbrace (x_1: a|k_0|k_1| \ldots |k_m)  \mid & x_1 \in \lbrace \alpha_1^{\pm}, \beta_1^{\pm}, \gamma_1^{\pm}\rbrace, m \in \mathbb{N} \text{ and }  a, k_i \in \mathbb{N}_{\geq 1} \text{ for all } 0 \leq i \leq m \\
&  \text { and } k_m \geq 2 \text{ if } m \geq 1 \rbrace 
\cup \lbrace(x_1:0|\pm 1)|x_1 \in \lbrace \alpha_1, \beta_1, \gamma_1 \rbrace \rbrace
\end{align*}
and let $C$ be an $E_{\Lambda}$-rigid string given by a sequence $(x_1:a_1,a_2,\ldots ,a_n)$. We identify $C$ with a unique element in $\Psi$. For $x_1$ in $\lbrace \alpha_1 , \beta_1, \gamma_1 \rbrace$ we identify the negative simple representation $\mathcal{S}^-_i$ where $i \notin \lbrace s(x_1),t(x_1) \rbrace$ with $(x_1:0|-1)\in\Psi$ and the string of length zero at $s(x_1)$, which was denoted by $(x_1:0)$ before, with $(x_1:0|1)\in\Psi$.

Now we deal with the more general cases and we drop the arrow $x_1 \in \lbrace \alpha_1^{\pm}, \beta_1^{\pm}, \gamma_1^{\pm}\rbrace$ in our notation. 

We define the numbers $k_i$ and subsequences $v_i=v_i^C$, $w_i=w_i^C$, and $P_{i-1}=P_{i-1}^C$ of $(a_1,a_2,\ldots ,a_n)$ inductively for all $0 \leq i \leq m$. These subsequences obviously depend on $C$, but we will omit the superscript $C$ to simplify notation. They will satisfy the following conditions:

\begin{compactenum}[(i)]
\item $v_i=(a_1,\ldots, a_k)$ for some $0\leq k \leq n$ is symmetric and thus corresponds to an $E_{\Lambda}$-rigid string. Furthermore, we will have $v_m=(a_1,\ldots,a_n)=C$.
\item $w_i=(a_1,\ldots, a_k +1,)$ differs from $v_i$ only in the last entry and thus corresponds to a strongly reduced band. Also note that $w^{\sigma}_i=(,a_1+1,a_2,\ldots,a_k)$.
\item $P_{-1}=(,a+1,)$ and $P_{i-1}=P_{i-2}w_{i-1}$ for $1\leq i \leq m$ is symmetric. Hence we have $P_{i-1}=w^{\sigma}_{i-1}P_{i-2}$.
\item If $v_i$ appears as a subsequence in $(a_1,\ldots, a_n)$ it can only be followed by the subsequence $P_{i-1}v_i$ or $P_{i-1}w_i$.
\item The number $k_i$ is the maximal number of consecutive strings $v_{i-1}$ in $C$, that is we cannot have
\[C=(\ldots \underbrace{v_{i-1}P_{i-2}v_{i-1}\ldots P_{i-2}v_{i-1}}_{k} \ldots)\]
with $k>k_i$. On the other hand $k_i -1$ is minimal in the sense, that we cannot have
\[C=(\ldots P_{i-1} \underbrace{v_{i-1}P_{i-2}v_{i-1}\ldots P_{i-2}v_{i-1}}_{k}P_{i-1} \ldots)\] with $k<k_i -1$.
\end{compactenum}
We begin by setting $a:=a_1$ (thus we have $a_i \in \lbrace a, a+1 \rbrace$ for all $1 \leq i \leq n$ by Proposition \ref{rig} (ii)) and 
\[k_0=\max \lbrace 1\leq i\leq n \mid  a_j=a \text{ for all } 1\leq j\leq i \rbrace. \]
Furthermore, we define the sequences $v_0=(a_1,\ldots,a_{k_0})=(a,\ldots,a)$ and 
\newline $w_0=(a_1,\ldots,a_{k_0 -1},a+1)=(a,\ldots,a,a+1,)$ and $P_{-1}:=w_{-1}:=(,a+1,)$.

If $k_0=n$ we have $C=v_0$. Then we identify $C$ with $(a|k_0)$ and are done. Note that in that case $k_0$ can take any value in $\mathbb{N}_{\geq 1}$. 

Otherwise we know that 
$C=(v_0,a+1,\ldots)=(v_0P_{-1}\ldots)$
and by the definitions of $k_0$, $v_0$, $w_0$ and $P_{-1}$ and Proposition \ref{rig} it is easy to see, that conditions (i)-(v) are satisfied. This is the base case of our induction. Before proceeding with the general induction step, say from $i$ to $i+1$, we will have a closer look at the step from $0$ to $1$. For that case notation is relatively simple and the general step is then done analogously.  

We define $k_1$ to be the maximal number of consecutive strings $v_0$ in $C$ in the sense that
\[C=(\underbrace{v_0P_{-1}\ldots P_{-1}v_0}_{k_1 \text{ times } v_0},s)\]
 such that either $s=\emptyset$ or $s=P_{-1}w_0...$. We define sequences \[v_1=(\underbrace{v_0P_{-1}\ldots P_{-1}v_0}_{k_1 \text{ times } v_0}) \text{ and } w_1=(\underbrace{v_0P_{-1}\ldots P_{-1}v_0}_{k_1 -1 \text{ times } v_0}P_{-1}w_0) \text{ and } P_0=P_{-1}w_0. \] Then the conditions (i)-(iii) are satisfied by definition. 
 
Again there are two possible cases. If $s$ is the empty sequence, that is if $C=v_1$ we identify $C$ with the sequence $(a|k_0|k_1)$. The conditions (iv) and (v) are obviously satisfied. Note that in this case $k_1 \geq 2$ because of the symmetry of $C$.

Now, if we suppose that $s=P_{-1}w_0...$ it is possible, that $k_1=1$ and in that case we have $v_1=v_0$ and $w_1=w_0$. In any case, we have to check that $k_1$, $v_1$, $w_1$ and $P_{-1}$ satisfy conditions (iv) and (v). Since $v_1$ ends with $v_0$, we know by induction that it can only be followed by $P_{-1}v_0$ or $P_{-1}w_0=P_0$ in $C$. Suppose it is followed by $P_{-1}v_0$ at some point. Then we have
\[C=(v_1P_{-1}w_0 \ldots ,v_1P_{-1}v_0\ldots)\]
contradicting Proposition \ref{rig} (iii). 
Now if $v_1P_0$ appears as a subsequence in $C$, it cannot be at the very end, because of the symmetry of $C$. Let $s$ be the longest subsequence such that
\[C=(\ldots ,v_1P_0s \ldots ) \text{ and } v_1=(s\ldots).\]
If $s=v_1$ we are done. If $s$ is a strict subsequence of $v_1$, then we know by symmetry of $C$, that it cannot be at its very end and thus
  \[C=(\ldots ,v_1P_0s,b \ldots ) \text{ and } v_1=(s,b'\ldots)\] with $\lbrace b, b'\rbrace=\lbrace a,a+1 \rbrace$. If $b=a$ and $b'=a+1$ we have a contradiction with Proposition \ref{rig} (iii). Thus we have $b=a+1$ and $b'=a$. If $s,b,=w_1$ we are done. Otherwise we have $\left| s,a+1,\right|< \left|w_1 \right|$ and then because   
\[C=(v_1P_0\ldots ,v_1P_0s,a+1,\ldots)=(v_0P_{-1}\underbrace{v_0P_{-1}\ldots P_{-1}v_0P_{-1}w_0}_{w_1} \ldots w^{\sigma}_0P_{-1}s,a+1,\ldots)\]
a contradiction with Proposition \ref{rig} (iv).

It is easy to see that $k_1$ (respectively $k_1 -1$) satisfies the maximality (resp. minimality) condition in (v) because of Proposition \ref{rig} (iii) (respectively (iv)).

Now assume that the numbers $k_j$ and the sequences $v_j$, $w_j$ and $P_{j-1}$ have been defined for $0\leq j \leq i$ for some $1\leq i$, satisfying the conditions (i)-(v). If $C=v_i$ we are done. Note that in that case, we have $k_i\geq 2$ because of the symmetry of $C$.

Otherwise we define $k_{i+1}$ to be maximal number of consecutive strings $v_i$ in $C$ in the sense that
\[C=(\underbrace{v_iP_{i-1}\ldots P_{i-1}v_i}_{k_{i+1} \text{ times } v_i},s)\]
 such that either $s=\emptyset$ or $s=P_{i-1}w_i...$. We define sequences \[v_{i+1}=(\underbrace{v_iP_{i-1}\ldots P_{i-1}v_i}_{k_{i+1} \text{ times } v_i}) \text{ and } w_{i+1}=(\underbrace{v_iP_{i-1}\ldots P_{i-1}v_i}_{k_{i+1} -1 \text{ times } v_i}P_{i-1}w_i) \text{ and } P_i=P_{i-1}w_i. \]
We first check that $P_i$ is symmetric, then it follows by definition that $v_{i+1}$ (resp.  $w_{i+1}$) satisfies condition (i) (resp. condition (ii)). We have
\begin{align*}
P_{i} & = P_{i-1}w_i
 = P_{i-1} v_{i-1}P_{i-2} \ldots P_{i-2}v_{i-1}P_{i-1} \\
& = P_{i-1}^{\sigma} v_{i-1}^{\sigma}(P_{i-2})^{\sigma} \ldots(P_{i-2})^{\sigma}v_{i-1}^{\sigma}P_{i-1}
 =w_{i}^{\sigma}P_{i-1}=w_{i}^{\sigma}P_{i-1}^{\sigma}=(P_{i})^{\sigma}.
\end{align*}
where we used the definitions and the symmetry of $v_{i-1}$, $P_{i-2}$ and $P_{i-1}$.

Now, if $s=\emptyset$ we have $C=v_{i+1}$, the conditions (i)-(v) are satisfied and we are done. In this case we must have $k_{i+1}\geq 2$ by symmetry of $C$. 

On the other hand if we suppose that $s=P_{i-1}w_i...$ it is possible, that $k_{i+1}=1$ and in that case we have $v_{i+1}=v_i$ and $w_{i+1}=w_i$. In any case, we have to check that $k_{i+1}$, $v_{i+1}$, $w_{i+1}$ and $P_{i}$ satisfy conditions (iv) and (v). This is done by induction, in exactly the same way as for the case $i=0$ above. 

Since the sequence $C=(a_1,\ldots,a_n)$ is finite this process has to stop eventually, say we have $C=v_m$ for some $m \in \mathbb{N}$. In that case we have $k_m \geq 2$ by symmetry of $C$. Note, that the numbers $k_i$ do not depend on each other and hence any element in $\Psi$ corresponds to a unique $E_{\Lambda}$-rigid string. 

We will do a very similar construction for strongly reduced bands. In fact, any strongly reduced band $B$ will be associated with an element $(a|k_0|k_1|\ldots|k_m) \in \Psi$ where the numbers $k_i$ and sequences $v_i$, $w_i$ and $P_{i-1}$ will be defined inductively in the same way as for strings, such that $B$ is given by the sequence $w_m$. 

Let $B=(a_1,\ldots, a_n,)$ be a strongly reduced band. After possibly rotating $B$ we can assume that $a_1=a:=\min_{1 \leq i\leq n}\lbrace a_i \rbrace$. Then by Lemma \ref{formband}(i) we have $a_i \in \lbrace a,a+1 \rbrace$ for all $1\leq i \leq n$. If $n=1$ then we have $B=(a,)=(a-1|1)=w_0$ and we are done. Otherwise since $B$ is not a power of any other band and $a$ was chosen minimal, we have $a_i=a+1$ for some $1\leq i \leq n$. Hence we can choose 
\[ \widetilde{k_0} = \max \lbrace j \in \mathbb{N}_{\geq 1}\vert \exists \text{ a subsequence } a_{i+1},\ldots,a_{i+j} \text{ of } B^{\infty} \text{ with } a_{i+1}=\ldots =a_{i+j}=a \rbrace \]
and after possibly rotating we can assume that
\[B=(\underbrace{a,\ldots,a}_{\widetilde{k_0}} ,a+1,s).\]
If $s=\emptyset$ we set $k_0= \widetilde{k_0}+1$ and define $v_0$, $w_0$ and $P_{-1}$ correspondingly as before. Then $B=w_0$ and we are done.

Otherwise set $k_0=\widetilde{k_0}$ and define $v_0$, $w_0$ and $P_{-1}$ correspondingly. Now since $B$ is not a power of any other band and since $k_0$ was chosen maximal, we must have $P_{-1}w_0$ as a subsequence of $B^{\infty}$. Hence we can choose 
\[ \widetilde{k_1} := \max \lbrace j \in \mathbb{N}_{\geq 1}\vert \underbrace{v_0,a+1,v_0,\ldots,a+1,v_0}_{j\text{ copies of }v_0} \text{ is a subsequence of }  B^{\infty} \rbrace \]
and after possibly rotating we can assume
\[B=(\underbrace{v_0P_{-1}v_0\ldots P_{-1}v_0}_{\widetilde{k_1} }P_{-1}w_0s).\]
If $s=\emptyset$ set $k_1=\widetilde{k_1}+1 $ and define $v_1$, $w_1$ and $P_{0}$ correspondingly as before. Then $B=w_1=(a|k_0|k_1)$ and we are done.

Otherwise set $k_1=\widetilde{k_1}$ and define $v_1$, $w_1$ and $P_{0}$ correspondingly. Then $k_1$ satisfies the maximality condition by definition. We only need to show, that in this case $B^{\infty}$ satisfies the condition (iv). This can be done using the same arguments as for the $E_{\Lambda}$-rigid strings. Then we proceed inductively defining the numbers $k_i$. If $k_i$ and thus $v_i$,$w_i$ and $P_{i-1}$ are defined we know that since $B$ is not a power of any other band, at some point we must have $v_iP_{i-1}w_i$ as a subsequence of $B^{\infty}$. Hence we can define $k_{i+1}$ as the number satisfying the maximality condition. Then proceed in the same way as in the step from $0$ to $1$. Since the sequence $(a_1,\ldots,a_n,)$ is finite, this process has to stop eventually. 

In the following we will give some examples to make this notation more transparent. 
\begin{ex}
Let $C$ be an $E_{\Lambda}$-rigid string given by the sequence \[(a,a,a,a,a+1,a,a,a,a+1,a,a,a,a,a+1,a,a,a,a+1,a,a,a,a).\] Then
 $k_0=4$ and $v_0=a,a,a,a$ and $P_{-1}=,a+1,$, $k_1=1$ and $v_1=v_0$ and $P_0=,a+1,a,a,a,a+1,$, $k_2=3$ and $C=v_2=v_1P_0v_1P_0v_1=(a|4|1|3).$

Now let $C=(a,a,a,a+1,a,a,a,a+1,a,a,a,a+1,a,a,a,a+1,a,a,a)$ be an $E_{\Lambda}$-rigid string. Then
 $k_0=3$ and $v_0=a,a,a$ and $P_{-1}=,a+1,$ and $k_1=5$ and $C=v_1=(a|3|5)$.
 
Let $(a|1|1|1|2)$ be in $\Psi$. Then we have $v_0=a$ and $v_0=v_1=v_2$ and $w_{-1}=a+1,$ and $w_{-1}=w_0=w_1$ and hence \[C=v_3=v_2P_1v_2=(a,w_{-1}w_0w_1a)=(a,a+1,a+1,a+1,a).\]
\end{ex}
\subsection{Mutation of $E_{\Lambda}$-rigid components}
We are now ready to describe the structure of the graph $\Gamma$ explicitly. In the figures below we depicted all the neighbours of $\mathcal{S}_3^-$ and $\mathcal{S}_1$. The edges incident to the other simple and negative simple decorated representations are similar. To simplify the complicated picture we leave out the loops (at all the vertices). 
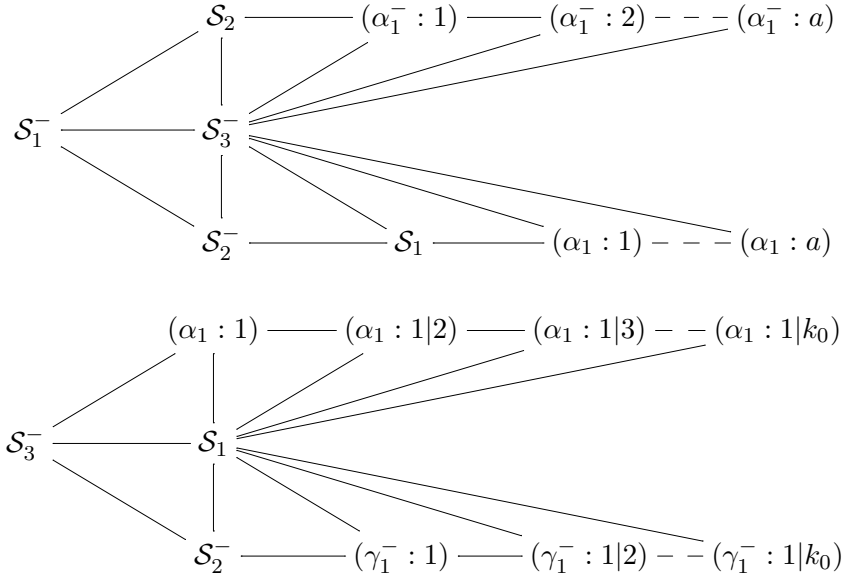
\[
\xygraph{!{<0cm,0cm>;<0cm,1.5cm>:<-2.5cm,0cm>::}
!{(2,1)}*+{\mathcal{S}_1^-}="1"
!{(1,0)}*+{\mathcal{S}_2^-}="2"
!{(2,0)}*+{\mathcal{S}_3^-}="3"
!{(3,0)}*+{\mathcal{S}_2}="4"
!{(1,-1)}*+{\mathcal{S}_1}="5"
!{(3,-1)}*+{(\alpha_1^-:1)}="7"
!{(1,-2)}*+{(\alpha_1:1)}="15"
!{(3,-2)}*+{(\alpha_1^-:2)}="17"
!{(1,-3)}*+{(\alpha_1:a)}="25"
!{(3,-3)}*+{(\alpha_1^-:a)}="27"
"1"-"4"
"3"-"7"
"3"-"17"
"3"-"27"
"1"-"2"
"1"-"3"
"2"-"3"
"3"-"5"
"3"-"15"
"3"-"25"
"3"-"4"
"2"-"5"
"4"-"7"
"5"-"15"
"7"-"17"
"15":@{--}"25"
"17":@{--}"27"
}
\]
\[
\xygraph{!{<0cm,0cm>;<0cm,1.5cm>:<-2.5cm,0cm>::}
!{(2,1)}*+{\mathcal{S}_3^-}="1"
!{(1,0)}*+{\mathcal{S}_2^-}="2"
!{(2,0)}*+{\mathcal{S}_1}="3"
!{(3,0)}*+{(\alpha_1:1)}="4"
!{(1,-1)}*+{(\gamma_1^-:1)}="5"
!{(3,-1)}*+{(\alpha_1:1|2)}="7"
!{(1,-2)}*+{(\gamma_1^-:1|2)}="15"
!{(3,-2)}*+{(\alpha_1:1|3)}="17"
!{(1,-3)}*+{(\gamma_1^-:1|k_0)}="25"
!{(3,-3)}*+{(\alpha_1:1|k_0)}="27"
"1"-"4"
"3"-"7"
"3"-"17"
"3"-"27"
"1"-"2"
"1"-"3"
"2"-"3"
"3"-"5"
"3"-"15"
"3"-"25"
"3"-"4"
"2"-"5"
"4"-"7"
"5"-"15"
"7"-"17"
"15":@{--}"25"
"17":@{--}"27"
}
\]
It is rather easy to see that this actually is a subgraph of $\Gamma$ and that these have to be all the edges incident to $\mathcal{S}_3^-$ and $\mathcal{S}_1$. In the case of $\mathcal{S}_3^-$ this follows directly from \cite[Lemma 5.4.]{CLS}. For $\mathcal{S}_1$ it follows by a case by case study, similar to the ones we have seen before and is therefore omitted here. Up to equivalence the only strings that can appear as neighbours of $\mathcal{S}_1$ are given by sequences of the form $(\gamma_1^-:1|k_0)$ and $(\alpha_1:1|k_0)$ for $k_0 \in \mathbb{N}$. Similar considerations show that the band given by the sequence $(\alpha_1:1,)$ is the only band connected with $\mathcal{S}_1$ via an edge.
\begin{lemma}\label{shorter}
Let $C=(x_1:a_1,\ldots, a_n)$ given by $(x_1: a|k_0| \ldots |k_{m-1} |k_m) \in \Psi$ with $a>0$ an $E_{\Lambda}$-rigid string. Then there are (precisely) two $E_{\Lambda}$-rigid strings $C_1=(x_1:b_1,\ldots, b_s)$ and $C_2=(x_1:b'_1,\ldots, b'_t)$ with $s,t \leq n$ which are neighbours of $C$ in $\Gamma$. These are 
\begin{align*}
C_1&= \begin{cases} 
(x_1: a| k_0| \ldots |k_{m-1} |k_m-1)  &\text{ if } k_m \geq 3 \text{ or if } m=0 \text{ and } k_0=2 \\
 (x_1:a|k_0|\ldots |k_{j})  &\text{ if } m\geq 1 \text{ and }k_m=2 \\ 
 (x_1:0|-1)  &\text{ if } m=0 \text{ and } k_0=1 
\end{cases}\\
\text{ where } &j:=\max\lbrace 1\leq i\leq m-1\mid k_i\neq 1\rbrace.\\
C_2&:= \begin{cases}
(x_1: a|k_0| \ldots |k_{m-1} +1) &\text{ if }m\geq 1 \\
(x_1: a-1|1) &\text{ if } m=0.
\end{cases}
\end{align*}
and in particular we have $s+t=n$ and if $n>1$, then we have
\[(a_1,\ldots,a_n)=(b_1,\ldots,b_s,b'_1+1,b'_2,\ldots,b'_t).\]
\end{lemma}
Now the description of the neighbours of a given $E_{\Lambda}$-rigid string is a direct consequence of Lemma \ref{shorter}.
\begin{prop}\label{neighbours}
Let $C=(x_1:a_1,\ldots, a_n)$ given by $(x_1: a|k_0| \ldots |k_{m-1} |k_m) \in \Psi$ with $a>0$ an $E_{\Lambda}$-rigid string and $C_1$ and $C_2$ as in Lemma \ref{shorter}. If we set
\begin{align*}
C_3&:=(x_1: a| k_0| \ldots |k_{m-1} |k_m+1),\\
C_4^j&:=(x_1: a|k_0| \ldots |k_{m-1} |k_m|\underbrace{1|\ldots|1}_j|2) \text{ with } j\in \mathbb{N},\\
C_5^{\ell}&:= \begin{cases} 
(x_1: a|k_0| \ldots |k_{m-1}|k_m-1|\ell)  &\text{ if } k_m \geq 2\\
(x_1:a+1|\ell-1)  &\text{ if } m=0 \text{ and } k_0=1 
\end{cases}\\
\text{ where } &\ell \in \mathbb{N}_{\geq 2}.
\end{align*}
then in $\Gamma$ we have the full subgraph
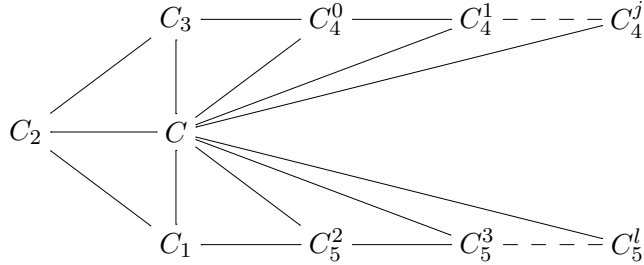
\[
\xygraph{!{<0cm,0cm>;<0cm,1.5cm>:<-2cm,0cm>::}
!{(2,1)}*+{C_2}="1"
!{(1,0)}*+{C_1}="2"
!{(2,0)}*+{C}="3"
!{(3,0)}*+{C_3}="4"
!{(1,-1)}*+{C_5^2}="5"
!{(3,-1)}*+{C_4^0}="7"
!{(1,-2)}*+{C_5^3}="15"
!{(3,-2)}*+{C_4^1}="17"
!{(1,-3)}*+{C_5^{l}}="25"
!{(3,-3)}*+{C_4^j}="27"
"1"-"4"
"3"-"7"
"3"-"17"
"3"-"27"
"1"-"2"
"1"-"3"
"2"-"3"
"3"-"5"
"3"-"15"
"3"-"25"
"3"-"4"
"2"-"5"
"4"-"7"
"5"-"15"
"7"-"17"
"15":@{--}"25"
"17":@{--}"27"
}
\]
and these are all $E_{\Lambda}$-rigid neighbours of $C$.
\end{prop}
\begin{prop}\label{mutation}
If $Z_1,Z_2$ are indecomposable $E_{\Lambda}$-rigid components in $\srdirr(\Lambda)$ that are connected by an edge in $\Gamma$, then there exist precisely two different indecomposable $E_{\Lambda}$-rigid components $Z_3,Z'_3 \in\srdirr(\Lambda)$ such that $\lbrace Z_1,Z_2,Z_3 \rbrace$ and $\lbrace Z_1,Z_2,Z'_3 \rbrace$ are $E_{\Lambda}$-rigid component clusters.
\end{prop}
\begin{proof}[Proof of Lemma \ref{shorter}]
We will first show that $C_1$ and $C_2$ are in fact neighbours of $C$ by checking the conditions in Corollary \ref{edge}. 

If $m=0$ then $C=(a_1,\ldots ,a_{k_0 -1},a_{k_0})=(a,\ldots,a,a)$ and we only have to check conditions (i) and (iv). Since we have $C_1=(a_1,\ldots ,a_{k_0 -1})=(a,\ldots,a)$ if $k_0>1$ and $C_2=(a-1)$ these conditions are obviously satisfied. 

So now suppose that $m\geq 1$. If $k_m\geq 3$ then $v_{m-1}=v_{m-1}^C=v^{C_1}_{m-1}$ and we have \[C=\underbrace{v_{m-1}P_{m-2} \ldots P_{m-2}v_{m-1}P_{m-2}v_{m-1}}_{k_m \text{ copies of }v_{m-1}} \quad\text{and}\quad   C_1=\underbrace{v_{m-1}P_{m-2} \ldots P_{m-2}v_{m-1}}_{k_m -1\text{ copies of }v_{m-1}} \]
and hence if $C_1=(s,a+1,\ldots)$ for some subsequence $s$, then also $C=(s,a+1,\ldots)$. Since $C_1$ is $E_{\Lambda}$-rigid we see Corollary \ref{edge}(ii)a) is satisfied. On the other hand suppose $C=(s,a+1,\ldots)$ and $C_1=(\ldots,s,a\ldots)$. Then we also have $C=(\ldots,s,a\ldots)$ contradicting that $C$ is $E_{\Lambda}$-rigid. 
If $C_1=(\ldots,a+1,s,a+1,\ldots)$ then also $C=(\ldots,a+1,s,a+1,\ldots)$ and since $C$ is $E_{\Lambda}$-rigid we see that Corollary \ref{edge}(iii)a) is satisfied. Now let $C=(\ldots,a+1,s,a+1,\ldots)$ and $C_1=(\ldots a,s,a\ldots)$. Then again we see that $C=(\ldots a,s,a \ldots)$ contradicting that $C$ is $E_{\Lambda}$-rigid. 
Finally we have to check if $C=(\ldots,C_1,\ldots)$. But, since $C_1$ begins with $v_{m-1}$ we know by the maximality of $k_{m-1}$ that the only copies of $C_1$ in $C$ are the ones at the very beginning and the end of $C$.

If $k_m=2$ then we have 
\[C=v_m=v_{m-1}P_{m-2}v_{m-1} \quad\text{and}\quad   C_1=v_j^{C_1}=v_j^{C} =v^C_{m-1}\]
where the last equality holds by choice of $j$. Now the same considerations as above show, that $C$ and $C_1$ satisfy the conditions in Corollary \ref{edge}.

Now if $m\geq 1$ then $v_{m-2}=v_{m-2}^C=v^{C_2}_{m-2}$ and we have 
\[C_2=v^{C_2}_{m-1}=v^C_{m-1}P_{m-3}v_{m-2}=v^C_{m-2}P_{m-3}v_{m-1}\]
and recall that $P_{m-2}=P_{m-3}w_{m-2}$ and $w_{m-2}$ and $v_{m-2}$ only differ in their last entry. Thus if $C_2=(s,a+1,\ldots)$, then we also have $C=(s,a+1,\ldots)$ and cannot have $C=(s,a+1,\ldots,s,a,\ldots)$ thus Corollary \ref{edge}(ii)a) is satisfied. If $C=(s,a+1,\ldots)$ and $\left|s,a+1\right|<\left|C_2\right|$ then we also have $C_2=(s,a+1,\ldots)$ (here $\left|C\right|$ is the length of the sequence describing $C$, i.e. if $C$ is given by $(x_1:a_1,\ldots,a_n)$ then $\left|C\right|=n$). If  $\left|s,a+1\right|\geq\left|C_2\right|$, then we cannot have $C_2=(\ldots,s,a\ldots)$ (the comma in front of the sequence $s$ is essential).

If $C_2=(\ldots,a+1,s,a+1,\ldots)$ then we also have $C=(\ldots,a+1,s,a+1,\ldots)$ and thus Corollary \ref{edge}(iii)a) is satisfied. If $C=(\ldots,a+1,s,a+1,\ldots)$ such that $\left|,a+1,s,a+1,\right|\leq \left|C_2\right|$ we can always write $s=s_1,s_2$ such that either $C_2=(\ldots ,a+1,s_1,s_2,a+1,\ldots)$ or $C_2=(\ldots ,a+1,s_1)$ or $C_2=(s_2,a+1,\ldots)$. In any case, it cannot happen that $C_2=(\ldots ,a,s_1,s_2,a,\ldots)$.
Finally, since $C_2$ begins with $v^C_{m-1}$ we know by the maximality of $k_{m-1}$ that there are no copies of $C_2$ in $C$.

Now let $C'=(a'_1,\ldots,a'_r)$ be an $E_{\Lambda}$-rigid neighbour of $C$ with $r\leq n$. Then $C'$ is given by some sequence $(a|k'_0|\ldots|k'_p)$ and we have to show, that $C'$ is either $C_1$ or $C_2$. The cases $m=0$ and $m=1$ are easily dealt with, so we assume $m\geq 2$. 

Now if $p \geq m$, we show by induction that $k'_i=k_i$ for all $0\leq i \leq n-1$.
Since $m>0$ by definition we have 
\[ C=(\underbrace{a,\ldots,a}_{k_0} ,a+1,\ldots,a+1,\underbrace{a,\ldots,a}_{k_0 -1} ,a)\] and  \[C'=(\underbrace{a,\ldots,a}_{k'_0} ,a+1,\ldots,a+1,\underbrace{a,\ldots,a}_{k'_0 -1} ,a).\]
By Corollary \ref{edge}(ii) we get $k_0=k_0'$. Now let $1 \leq i \leq m-1$ and assume $k_j=k'_j$ for all $0 \leq j \leq i-1$. Then $v_{i-1}=v^C_{i-1}=v^{C'}_{i-1}$ and since $m>i$ by definition we have
\[ C=\underbrace{v_{i-1}P_{i-2}\ldots P_{i-2}v_{i-1}}_{k_i \text{ copies of } v_{i-1}} P_{i-2}v_{i-1} \underbrace{v_{i-1}P_{i-2}\ldots P_{i-2}v_{i-1}}_{k_i -1\text{ copies of } v_{i-1}} P_{i-2}v_{i-1}\]
and 
\[ C'=\underbrace{v_{i-1}P_{i-2}\ldots P_{i-2}v_{i-1}}_{k'_i \text{ copies of } v_{i-1}} P_{i-1}\ldots \underbrace{v_{i-1}P_{i-2}\ldots P_{i-2}v_{i-1}}_{k'_i -1\text{ copies of } v_{i-1}} P_{i-1}\ldots.\]
Since $P_{i-1}=P_{i-2}w_{i-1}$ we have $k_i=k'_i$ (using Corollary \ref{edge}(ii)). Since we assume, that $r\leq n$ we must have $k'_m<k_m$ and $p=m$ and we see by Corollary \ref{edge}(iv) that $C'=C_1$. 

If we assume that now $p < m$ then by the same argument as before we have $k'_i=k_i$ for all $0\leq i \leq p-1$. Considering the different cases $k'_p < k_p$, $k'_p=k_p$ and $k'_p>k_p$ we get a contradiction in the first case, $C'=C_1$ in the second and $C'=C_2$ in the third. 
\end{proof}
\subsection{Connectedness of the graph}
\begin{lemma}\label{bandstring}
Let $B$ be a strongly reduced band given by the sequence $(x_1:a|k_0|\ldots|k_m)$. Then there is an edge in $\Gamma$ between $B$ and the string encoded by the same sequence.  
\end{lemma}
\begin{proof}
In the following we will write $B^{\infty}$ to imply that we consider a sufficiently large number of copies of the sequence $(a_1,\ldots,a_n,)$. This will be denoted by
\[B^{\infty}=\ldots a_n,|a_1,a_2,\ldots,a_n,|a_1,\ldots.\]

Assume Corollary \ref{edgeband}(i) is violated. Hence we have $C=(s,a+1,\ldots)$ and for some positive integer $k$ we have $,s,a,$ as a subsequence in $B^{k}$. Now if this subsequence was contained in only one copy of $B$ it would also be a subsequence of $C$. This is impossible since $C$ is $E_{\Lambda}$-rigid. Hence we can consider
\[B^3=|\ldots,a+1,|s,a+1,\ldots ,\underbrace{s_1,a+1,|s_2}_s,a, \ldots ,a+1,|. \]
In that case we have \[C=(\underbrace{\boldsymbol{ s_1,a+1},s_2}_s ,a+1,\ldots \boldsymbol{,s_1,a})\] contradicting $C$ being $E_{\Lambda}$-rigid. 

That $C$ and $B^{k}$ satisfy condition (ii) for every positive integer $k$ is easily verified. If $a+1,s,a+1$ is a subsequence in $C$ it also is a subsequence in $B^{k}$. Now assume that $a+1,s,a+1$ is a subsequence of $B^{2}$ such that 
 \[B^2=|\ldots,a+1,\underbrace{s_1,|s_2}_s,a+1,\ldots ,a+1,| \]
 and $C=(\ldots,a,s,a,\ldots)$. Then either $B=(\ldots a,s,a,\ldots)$ which contradicts $B$ being strongly reduced or $C=(\boldsymbol{ s_2,a+1},\ldots,a,s_1,\boldsymbol{ s_2,a})$ which contradicts $C$ being $E_{\Lambda}$-rigid.
 
Finally assume that for some positive integer $k$ a copy of $C$ is a subsequence of $B^{k}$. Then we have
\[B^2=|\ldots,\underbrace{P,a+1,|Q}_C,\ldots |. \]
and therefore
$ C = (P,a+1,\ldots,P,a)$ which is again a contradiction. Hence there is an edge between $C$ and $B$ in $\Gamma(\Lambda)$.
\end{proof}
\begin{theorem}\label{connectec}
The graph of strongly reduced components is connected. The full subgraph on the $E_{\Lambda}$-rigid components is also connected. 
\end{theorem}
\begin{proof}
By Theorem \ref{neighbours} it follows that any $E_{\Lambda}$-rigid string is in the same component as the negative simple ones. Hence the second part of the Theorem follows. By Lemma \ref{bandstring} any strongly reduced band is connected with an $E_{\Lambda}$-rigid string. 
\end{proof}
\begin{remark}\label{cluster}
Let $Q$ be the Markov quiver and $\mathcal{A}_Q$ the corresponding cluster algebra. Then the cluster monomials $\mathcal{M}_Q$ of $\mathcal{A}_Q$ are precisely the generic Caldero-Chapoton functions of indecomposable $E_{\Lambda}$-rigid components in $\srdirr(\Lambda)$. This follows from the construction in \cite[Corollary 5.3]{DWZ2} of cluster variables as Caldero-Chapoton functions. These representations are gained by iterated mutation from the negative simple representations of $\Lambda$. These representations are all $E_{\Lambda}$-rigid, since $E_{\Lambda}(\M)$ is invariant under mutation (\cite[Theorem 7.1]{DWZ2}). On the other hand by  Proposition \ref{mutation} and Theorem \ref{connectec} it follows that the decorated representations obtained from mutation of the negative simples coincide with the family of indecomposable $E_{\Lambda}$-rigid decorated representations.
\end{remark}
\section{Component clusters}
\begin{lemma}\label{bandnei}
Let $B$ be a strongly reduced band given by the sequence $(a|k_0|\ldots|k_m)$. Then the $E_{\Lambda}$-rigid string $C$ given by the same sequence is the only neighbour of $B$ in $\Gamma$.
\end{lemma}
\begin{proof}
Let $(a|k'_0|\ldots|k'_p)\in \Psi $ differing from $(a|k_0|\ldots|k_m)$ in at least one entry and let $C'$ (respectively $B'$) be the corresponding string (respectively band). Assume that $p\geq m$ (the other case is similar) and let $i=\min \lbrace i \in \mathbb{N}\mid k_i \neq k'_i \rbrace$. Then $v_{i-1}=v_{i-1}^C=v_{i-1}^B$. We will show, that if $k_i < k'_i$, there exists a sequence $s$ such that
\begin{align*}
B^{\infty}&= \ldots P_{i-1}sP_{i-1}\ldots =\ldots a+1,\tilde{s},a+1\ldots\\
 B'^{\infty}&= \ldots v_{i-1}P_{i-2}sP_{i-1}v_{i-1}\ldots =\ldots ,a,\tilde{s},a,\ldots\\
 C&= v_{i-1}P_{i-2}sP_{i-2}v_{i-1}\ldots = (a,\tilde{s},a\ldots)
\end{align*}
and hence there is a non-trivial homomorphism from $B$ to $B'$ and from $B$ to $C'$. The case $k_i > k'_i$ is dealt with similarly. For simplicity we will assume that $k_i=2$ and $k'_i=3$. 

First assume that $i <m \leq p$, then we have
\begin{align*}
B^{\infty}&= \ldots P_{i-1}|v_{i-1}P_{i-2}v_{i-1}P_{i-1}\ldots \boldsymbol{P_{i-1}v_{i-1}P_{i-1}}|v_{i-1}\ldots\\
 B'^{\infty}&= \ldots P_{i-1}|\boldsymbol{v_{i-1}P_{i-2}v_{i-1}P_{i-2}v_{i-1}}P_{i-1}\ldots\\
 C&= (\boldsymbol{v_{i-1}P_{i-2}v_{i-1}P_{i-2}v_{i-1}}\ldots) 
\end{align*}
If $i=m=p$ we have
\begin{align*}
B^{\infty}&= \ldots\boldsymbol{P_{m-1}|v_{m-1}P_{m-1}|v_{m-1}P_{m-1}}|v_{m-1}\ldots\\
 B'^{\infty}&= \ldots P_{m-1}|\boldsymbol{v_{m-1}P_{m-2}v_{m-1}P_{m-1}|v_{m-1}P_{m-2}}v_{m-1}P_{m-1}|\ldots\\
 C&= (\boldsymbol{v_{m-1}P_{m-2}v_{m-1}P_{m-2}v_{m-1}}) 
\end{align*}
and if $m<i\leq p$ and we assume that $k_m=2=k'_m$ we have 
\begin{align*}
B^{\infty}&= \ldots\boldsymbol{P_{m-1}|v_{m-1}P_{m-1}|(v_{m-1}P_{m-1}|)^x v_{m-1}P_{m-1}}\ldots\\
 B'^{\infty}&= \ldots P_{m-1}|\boldsymbol{v_{m-1}P_{m-2}v_{m-1}P_{m-1}(v_{m-1}P_{m-1})^x v_{m-1}P_{m-2}v_{m-1}}\ldots\\
 C&= (\boldsymbol{v_{m-1}P_{m-2}v_{m-1}P_{m-1}(v_{m-1}P_{m-1})^x v_{m-1}P_{m-2}v_{m-1}}\ldots) 
\end{align*}
where $x$ depends on $k'_{m+1},\ldots,k'_p$. If $k'_{m+1}\geq 2$, then $x=0$ and $(v_{m-1}P_{m-1})^0$ is just the empty sequence. If $k'_{m+1}= 1$, we have to use that then $v^C_{m+1}=v^C_m=v^B_m$. 
\end{proof}
The next theorem is a direct consequence of Lemma \ref{bandstring} and Lemma \ref{bandnei}.
\begin{theorem}
There is a bijection between the $E_{\Lambda}$-rigid indecomposable components $Z \in \dirr^{\sr}(\Lambda)$ and the non $E_{\Lambda}$-rigid indecomposable components $\overline{Z} \in \dirr^{\sr}(\Lambda)$, such that if $Z$ and $\overline{Z}$ are identified under this bijection then $\lbrace Z, \overline{Z} \rbrace $ is a component cluster and these are all non $E_{\Lambda}$-rigid component clusters.
\end{theorem}
\begin{corollary}
\begin{compactenum}[(i)]
\item The $E_{\Lambda}$-rigid component clusters of $\Lambda$ are exactly the component clusters of cardinality $3$.
\item The non $E_{\Lambda}$-rigid component clusters of $\Lambda$ are exactly the component clusters of cardinality $2$.
\end{compactenum}
\end{corollary}
\section{The $\bf{g}$-vectors}
In this section we determine the generic $\bf{g}$-vectors of the strongly reduced indecomposable components for $\Lambda$ using Proposition \ref{pro}. 
Let $I_i \in \rep(\Lambda_p)$ be the injective envelope of the simple representation $S_i$ for $1\leq i \leq 3$. It is known that $I_i$ can be described by the paths in $Q$ that are ending in the vertex $i$. 
Let us first determine the $\bf{g}$-vectors for the string module (respectively band module), where the string (respectively band), is given by a sequence $(x_1:0)$ (respectively $(x_1:0,)$) for $x_1 \in \lbrace \alpha_1,\beta_1,\gamma_1 \rbrace$. Let $p=2$. If $x_1=\alpha_1$ the string module under consideration is the simple module $S_1$. Then the sequence
\[0\longrightarrow S_1 \overset{f}{\longrightarrow} I_0^{\Lambda_p}(S_1) {\longrightarrow} \Coker(f)\]
can be visualized in the diagrams
\[ \xygraph{!{<0cm,0cm>;<0.6cm,0cm>:<0cm,0.9cm>::}
!{(-5,0)}*+{1}="1"
!{(-4,0)}="6"
!{(-2,0)}="7"
!{(-2,2)}*+{2}="3"
!{(-1,1)}*+{3}="4"
!{(0,0)}*+{1}="5"
!{(2,2)}*+{2}="12"
!{(1,1)}*+{3}="11"
!{(3,1.5)}="16"
!{(5,1.5)}="17"
!{(6,2)}*+{2}="13"
!{(7,1)}*+{3}="14"
!{(8,1)}*+{\oplus}="15"
!{(10,2)}*+{2}="112"
!{(9,1)}*+{3}="111"
"6":"7"^{f}
"16":"17"
"3":"4"_{\beta_2}
"4":"5"_{\gamma_1}
"12":"11"^{\beta_1}
"11":"5"^{\gamma_2}
"13":"14"_{\beta_2}
"112":"111"^{\beta_1}
}
\]
and $g_{\Lambda}(S_1)=(-1,0,2)$. Similarly we find $g_{\Lambda}(S_2)=(2,-1,0)$ and $g_{\Lambda}(S_3)=(0,2,-1)$. The band module in this case is given by the band $B=\alpha_2^- \alpha_1$. In the following we choose $\lambda \in K^*$ and set $M(B):=M(B,\lambda,1)$. The sequence 
\[0\longrightarrow M(B) \overset{f}{\longrightarrow} I_0^{\Lambda_p}(M(B)) {\longrightarrow} \Coker(f)\]
can be visualized in the diagrams
\[ \xygraph{!{<0cm,0cm>;<0.6cm,0cm>:<0cm,0.9cm>::}
!{(-6,1)}*+{1}="0"
!{(-5,0)}*+{2}="1"
!{(-4,0)}="6"
!{(-2,0)}="7"
!{(-2,2)}*+{3}="3"
!{(-1,1)}*+{1}="4"
!{(0,0)}*+{2}="5"
!{(2,2)}*+{3}="12"
!{(1,1)}*+{1}="11"
!{(3,1.5)}="16"
!{(5,1.5)}="17"
!{(6,2)}*+{3}="13"
!{(7,1)}*+{1}="14"
!{(8,2)}*+{3}="112"
"0":"1"_{\alpha_1}
"0":@/^0.3cm/"1"^{\alpha_2}
"6":"7"^{f}
"16":"17"
"3":"4"_{\gamma_2}
"4":"5"_{\alpha_1}
"12":"11"^{\gamma_1}
"11":"5"^{\alpha_2}
"13":"14"_{\gamma_2}
"112":"14"^{\gamma_1}
}
\]
and $g_{\Lambda}(M(B))=(1,-1,0)$. Similarly we find that $g_{\Lambda}(M(\beta_1:0,))=(0,1,-1)$ and $g_{\Lambda}(M(\gamma_1:0,))=(-1,0,1)$.
\begin{prop}\label{g}
Let $Z \in \srirr$ be a strongly reduced indecomposable component corresponding to a string $C$ or band $B$ and let $g_{\Lambda}=(g_1,g_2,g_3)$ be its generic $\bf{g}$-vector. 
\begin{compactenum}[(i)]
\item If $C=(x_1:a_1,\ldots,a_n)$ with $a_1\geq 1$ set $a:=a_1+\cdots +a_n$. Then 
\begin{itemize}
\item if $x_1 =\alpha_1$ we have
$g_{\Lambda}=(-n+a,-a,n+1)$;
\item if $x_1 =\beta_1$ we have
$g_{\Lambda}=(n+1,-n+a,-a)$;
\item if $x_1 =\gamma_1$ we have
$g_{\Lambda}=(-a,n+1,-n+a)$;
\item if $x_1= \alpha_1^-$ we have
 $g_{\Lambda}=(2+a,n-2-a,-n+1)$;
\item if $x_1= \beta_1^-$ we have
 $g_{\Lambda}=(-n+1,2+a,n-2-a)$;
\item if $x_1= \gamma_1^-$ we have
 $g_{\Lambda}=(n-2-a,-n+1,2+a)$.
\end{itemize} 
 \item If $B=(x_1:a_1,\ldots,a_n,)$ with $a_1\geq 1$ set $a:=a_1+\cdots +a_n$. Then
 \begin{itemize}
\item if $x_1 =\alpha_1$ we have
$g_{\Lambda}=(-n+a,-a,n)$;
\item if $x_1 =\beta_1$ we have
$g_{\Lambda}=(n,-n+a,-a)$;
\item if $x_1 =\gamma_1$ we have
$g_{\Lambda}=(-a,n,-n+a)$;
\item if $x_1= \alpha_1^-$ we have
 $g_{\Lambda}=(a,n-a,-n)$;
\item if $x_1= \beta_1^-$ we have
 $g_{\Lambda}=(-n,a,n-a)$;
\item if $x_1= \gamma_1^-$ we have
 $g_{\Lambda}=(n-a,-n,a)$.
\end{itemize}  
\end{compactenum}  
\end{prop}
\begin{proof}
The Jordan-H\"older multiplicity $[\soc M(C):S_i]$ equals the number of appearances of the vertex $i$ in the socle of the diagram of $C$.
The diagram of $C$ is 
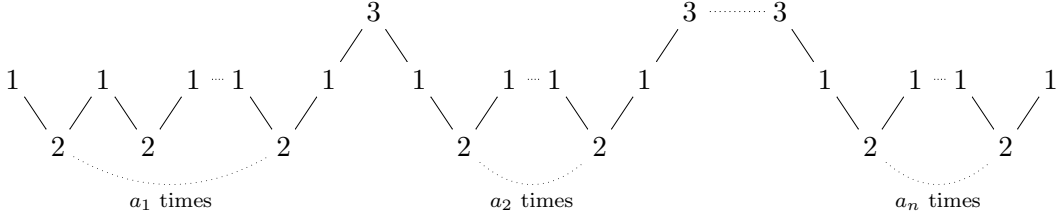
\[
\xygraph{!{<0cm,0cm>;<0.6cm,0cm>:<0cm,0.9cm>::}
!{(0,0)}*+{1}="1"
!{(1,-1)}*+{2}="2"
!{(2,0)}*+{1}="3"
!{(3,-1)}*+{2}="32"
!{(4,0)}*+{1}="33"
!{(5,0)}*+{1}="4"
!{(6,-1)}*+{2}="5"
!{(7,0)}*+{1}="6"
!{(8,1)}*+{3}="7"
!{(9,0)}*+{1}="8"
!{(10,-1)}*+{2}="9"
!{(11,0)}*+{1}="10"
!{(12,0)}*+{1}="11"
!{(13,-1)}*+{2}="12"
!{(14,0)}*+{1}="13"
!{(15,1)}*+{3}="14"
!{(17,1)}*+{3}="17"
!{(18,0)}*+{1}="18"
!{(19,-1)}*+{2}="19"
!{(20,0)}*+{1}="20"
!{(21,0)}*+{1}="21"
!{(22,-1)}*+{2}="22"
!{(23,0)}*+{1}="23"
"1"-"2"
"2"-"3"
"3"-"32"
"32"-"33"
"2":@/_0.5cm/@{.}_{a_1 \text{ times}}"5"
"33":@{.}"4"
"4"-"5"
"5"-"6"
"6"-"7"
"7"-"8"
"8"-"9"
"9":@/_0.5cm/@{.}_{a_2 \text{ times}}"12"
"9"-"10"
"10":@{.}"11"
"11"-"12"
"12"-"13"
"13"-"14"
"14":@{.}"17"
"17"-"18"
"18"-"19"
"19"-"20"
"19":@/_0.5cm/@{.}_{a_n \text{ times}}"22"
"20":@{.}"21"
"21"-"22"
"22"-"23"
}
\]
and thus $[\soc(M(C)):S_1]=[\soc(M(C)):S_3]=0$ and $[\soc(M(C)):S_2]=(a_1+\cdots +a_n)$.
Therefore for $p> \text{nil}_{\Lambda}(M(C))=3$ in the short exact sequence \[0\longrightarrow M(C) \overset{f}{\longrightarrow} I_0^{\Lambda_p}(M(C)) {\longrightarrow} \Coker(f)\]
we have
\[I^{\Lambda_p}_0(M(C)) \cong I_2^{a_1+\cdots +a_n}\]
and it is easily seen that \[ [\soc(\Coker(f)):S_1]=\sum_{i=1}^n(a_i-1) \qquad \text{and} \qquad [\soc(\Coker(f)):S_3]=n+1.\]
This concludes the proof of part (i).

Similarly, we can see that if $B$ is a band corresponding to a sequence $(\alpha_1:a_1,\ldots,a_n,)$ then $[\soc(M(B)):S_1]=[\soc(M(B)):S_1]=0$ and $[\soc(M(B)):S_2]=(a_1+\cdots +a_n)$.
Therefore we have 
\[I^{\Lambda_p}_0(M(B)) \cong I_2^{a_1+\cdots +a_n}\]
and it is easily seen that \[ [\soc(\Coker(f)):S_1]=\sum_{i=1}^n(a_i-1) \qquad \text{and} \qquad [\soc(\Coker(f)):S_3]=n.\]
The other cases are symmetric or are proven in a similar way.
\end{proof}
Note that the $\bf{g}$-vectors corresponding to the $E_{\Lambda}$-rigid strongly reduced indecomposable components all lie in the plane $x+y+z=1$. The $\bf{g}$-vectors corresponding to the non $E_{\Lambda}$-rigid strongly reduced components all lie in the plane $x+y+z=0$. A more detailed description of the $\bf{g}$-vectors will be given in the next section.

\section{The finite-dimensional Jacobian algebra}
In this section we want to deduce the graph of strongly reduced components of $\Lambda'$ from our previous description of the graph of strongly reduced components of $\Lambda$. We will show that all strongly reduced components of $\Lambda'$ arise from strongly reduced components of $\Lambda$ or their Auslander-Reiten translates. 
More precisely, let $Z \in \srirr$ indecomposable. Then $Z$ corresponds to a string described by a sequence $(x_1:a_1,\ldots,a_n)$ (respectively a band described by a sequence $(x_1:a_1,\ldots,a_n,)$). Hence the corresponding string module $M$ (respectively the 1-parameter family of band modules $M$) satisfies $\nil_{\Lambda}(M) \leq 3$. Now since $\Lambda'_5=\Lambda_5$ we can consider $M$ as a module over $\Lambda'$ and by Prop. \ref{pro} we see that
\[E_{\Lambda'}(M)= E_{\Lambda}(M).\]
By abuse of notation we will denote the corresponding indecomposable strongly reduced component in $\srirrp$ also by $Z$. Define its Auslander-Reiten translate by \[\tau_{\Lambda'}Z=\overline{\lbrace \tau_{\Lambda'}(M)\mid M \in Z \rbrace }\] and we claim that this is again an indecomposable strongly reduced component in $\srirrp$. Since $\Lambda'$ is finite-dimensional we have $\nil_{\Lambda'}(\tau_{\Lambda'}(M))\leq \dim \Lambda'$ and thus can apply Proposition \ref{pro} to obtain
\begin{align*}
E_{\Lambda'}(\tau_{\Lambda'}(M))&=\dim \Hom_{\Lambda'} (\tau^{-1}_{\Lambda'}(\tau_{\Lambda'}(M)),\tau_{\Lambda'}(M))\\
&=\dim \Hom_{\Lambda'} (M,\tau_{\Lambda'}(M))\\
&=\dim \Hom_{\Lambda'} (\tau_{\Lambda'}^{-1}(M),M)=E_{\Lambda'}(M)
\end{align*}
where the third equality follows from Lemma \ref{plam}. Now since $\Lambda'$ is tame, the claim follows. 
For the negative simple representations it does not make sense to apply the Auslander-Reiten translate. However, in this context, it fits perfectly to define $\tau_{\Lambda'}(\mathcal{S}^-_i)=I_i$  for $i=1,2,3$ because of the following reasoning: Let $M \in \mo (\Lambda)$ be $E_{\Lambda}$-rigid. Then $M\oplus \mathcal{S}^-_i$ is $E_{\Lambda'}$-rigid if and only if $\dim(M_i)=0$. This is equivalent to $\dim \Hom_{\Lambda'} (M,I_i)=0$, which in turn is equivalent to $\tau_{\Lambda'}M \oplus I_i$ being $E_{\Lambda'}$-rigid. The above discussion yields the following proposition.
\begin{prop}\label{X}
Let $Z \in \srdirr(\Lambda)$ a strongly reduced component. Then $Z$ is a strongly reduced component in $\srdirr(\Lambda')$. Moreover, the component defined by  
\[\tau_{\Lambda'}Z=\overline{\lbrace \tau_{\Lambda'}(M)\mid M \in Z \rbrace }\]
is a strongly reduced component in $\srdirr(\Lambda')$.
\end{prop}

\begin{prop}\label{gneu}
Let $Z \in \srirr$ be an $E_{\Lambda}$-rigid indecomposable component corresponding to a string $C$ given by a sequence $(x_1:a_1,\ldots,a_n)$ with $a_1\geq 1$ and set $a:=a_1+\cdots +a_n$. Let $\tau_{\Lambda'}Z$ be the $E_{\Lambda'}$-rigid component corresponding to $\tau_{\Lambda'}C$ and let $g_{\Lambda'}=(g_1,g_2,g_3)$ be its generic $\bf{g}$-vector. Then the following hold:
\begin{itemize}
\item if $x_1 =\alpha_1$ then we have
$g_{\Lambda'}=(-n+2+a,-a -2,n-1)$;
\item if $x_1 =\beta_1$ then we have
$g_{\Lambda'}=(n-1,-n+2+a,-a -2)$;
\item if $x_1 =\gamma_1$ then we have
$g_{\Lambda'}=(-a -2,n-1,-n+2+a)$;
\item if $x_1= \alpha_1^-$ we have
 $g_{\Lambda'}=(a,-n-1,-a+n)$;
\item if $x_1= \beta_1^-$ we have
 $g_{\Lambda'}=(-a+n,a,-n-1)$;
\item if $x_1= \gamma_1^-$ we have
 $g_{\Lambda'}=(-n-1,-a+n,a)$.
\end{itemize} 
\end{prop}
\begin{proof}
Computing a minimal projective presentation of $M(C)$ over $\Lambda'$ and then using Remark \ref{proj} yields the result. The computation of the projective presentation is done in a similar way, as the injective presentation in the proof of Theorem \ref{g}. The indecomposable projective $\Lambda'$-modules can be found in \cite{P}.
\end{proof}
In fact, all strongly reduced components of $\Lambda'$ arise in the way described in Prop. \ref{X}. We want to show this using the parametrization in Theorem \ref{gvectors} of the strongly reduced components by their generic $\bf{g}$-vectors. We set
\[\mathcal{G}= \lbrace g_{\Lambda'}(Z) \mid Z \in \srdirr(\Lambda') \text{ indecompasable} \rbrace \] and call two vectors $g_{\Lambda'}(Z)$ and $g_{\Lambda'}(Z')$ in $\mathcal{G}$ \textit{compatible}, if $E_{\Lambda'}(Z,Z')=0$. Our aim is to show that any vector in $\mathbb{Z}^3$ is a nonnegative $\mathbb{Z}$-linear combination of pairwise compatible $\bf{g}$-vectors in $\mathcal{G}$.
This follows from a connection between our work and work by Nathan Reading on universal geometric cluster algebras.
In order to illustrate this connection we recall some of Readings definitions and 
results. We refer to \cite{R2} for details on universal geometric cluster algebras in general, to \cite{R3} for universal geometric cluster algebras from surfaces and to \cite{R1} for a detailed construction of the universal geometric coefficients for the once-punctured torus.

Denote by $({\bf S},p)$ the once-punctured torus,
and let \[ B =B_Q= \left( \begin{array}{ccc}
0 & -2 & 2 \\
2 &  0 & -2 \\
-2 & 2 & 0 \end{array} \right),\] be a skew-symmetric exchange matrix (associated to a triangulation of $({\bf S},p)$) and $\mathcal{A}_B=\mathcal{A}_Q$ the corresponding cluster algebra. In \cite{R1} Reading describes a \textit{positive} $\mathbb{Z}$-\textit{basis for} $B^T$ (a ``mutation-linear'' analogue of the usual notion of a basis) in terms of Farey points and thus by \cite[Theorem 4.4]{R2} universal geometric coefficients for $B^T$.

A \textit{Farey point} is a pair of integers $(a,b)$ such that $\gcd(a,b)=1$ with the appropriate conventions for the $\gcd$. A \textit{standard Farey point} is a Farey point $(a,b)$ such that $a\geq 0$ and $b=1$ if $a=0$. Two (standard) Farey points $(a,b)$ and $(c,d)$ are called \textit{(standard) Farey neighbours} if $ad-bc=\pm 1$. A \textit{Farey triple} consists of three Farey points which are pairwise 
Farey neighbours. 

Let $T_0$ be a triangulation of $({\bf S},p)$ such that $B(T_0)$ is the transpose $B^T$. Reading shows in \cite[Corollary 3.2]{R1} that a positive $\mathbb{Z}$-basis for $B^T$ can be constructed as the shear coordinates of allowable curves in $({\bf S},p)$. He has the following result.
\begin{prop}\cite[Prop 5.1.]{R1}\label{reading}
The shear coordinates with respect to $T_0$ of allowable curves in $({\bf S},p)$ are as follows:
\begin{compactenum}[(i)]
\item The cyclic permutations of $(1-b,a+1,b-a-1)$ for Farey points $(a,b)$ with $a\geq 0$ and $b>0$ (corresponding to curves with counterclockwise spirals).
\item The cyclic permutations of $(-1-b,a-1,b-a+1)$ for Farey points $(a,b)$ with $a> 0$ and $b\geq 0$ (corresponding to curves with clockwise spirals).
\item The nonzero integer vectors $(x,y,z)$ with $x+y+z=0$ such that $x,y$ and $z$ have no common factors (corresponding to closed curves).
\end{compactenum} 
\end{prop}
Two allowable curves in $({\bf S},p)$ are \textit{compatible} if they do not intersect. Let $(a,b)$ be a standard Farey point. Following Readings notation we denote by $cl(a,b),cw(a,b)$ and $ccw(a,b)$ the associated allowable closed curve, curve with clockwise spirals and curve with counterclockwise spirals respectively.
\begin{prop}\cite[Prop 4.6.]{R1}
 The compatible pairs of allowable curves in $({\bf S},p)$ are:
\begin{compactenum}[(i)]
\item $cl(a,b)$ and $cw(a,b)$ for any standard Farey point $(a,b)$.
\item $cl(a,b)$ and $ccw(a,b)$ for any standard Farey point $(a,b)$.
\item $cw(a,b)$ and $cw(c,d)$ for standard Farey neighbours $(a,b)$ and $(c,d)$.
\item $ccw(a,b)$ and $ccw(c,d)$ for standard Farey neighbours $(a,b)$ and $(c,d)$.
\end{compactenum}
\end{prop} 
By \cite[Proposition 5.2.]{R3} the $\bf{g}$-vectors of cluster variables of $\mathcal{A}_B$ are given by shear coordinates of allowable curves, that are not closed and with counterclockwise spirals. Furthermore, two $\bf{g}$-vectors are \textit{compatible}, i.e. the corresponding cluster variables belong to the same cluster, if and only if the the corresponding curves are compatible. He thus recovers a result by N\'{a}jera \cite{N}.
\begin{prop}\cite[Corollary 7.3.]{R1}\label{Y}
The $\bf{g}$-vectors of cluster variables of $\mathcal{A}_B$ are the cyclic permutations of vectors $(1-b,a+1,b-a-1)$, for standard Farey points $(a,b)$. Two $\bf{g}$-vectors are compatible if and only if the corresponding Farey points are Farey neighbours. 
\end{prop}
\begin{corollary}
The set $\mathcal{G}$ is given by the vectors described in Proposition \ref{reading} and thus is a positive $\mathbb{Z}$-basis for $B^T$. Two generic $\bf{g}$-vectors in $\mathcal{G}$ are compatible if and only if the corresponding allowable curves are compatible.
\end{corollary}
\begin{proof}
By Remark \ref{cluster} we know that the subset $\mathcal{G'}$ of $\mathcal{G}$ given by the $\bf{g}$-vectors $g_{\Lambda'}(Z)$ where $Z \in \srdirr(\Lambda)$ is $E_{\Lambda}$-rigid and indecomposable is the set of $\bf{g}$-vectors of the cluster variables of $\mathcal{A}_B$. Hence by Proposition \ref{Y} they are precisely given by the vectors described in Proposition \ref{reading}(i). As the $\bf{g}$-vectors of the cluster variables are compatible if and only if the corresponding allowable curves are, the same holds for the $\bf{g}$-vectors in $\mathcal{G'}$. 

Let $Z$ be an $E_{\Lambda}$-rigid indecomposable strongly reduced component in $\srdirr(\Lambda)$ with $\bf{g}$-vector $(1-b,a+1,b-a-1)$. Let $Z'$ be the unique strongly reduced component in $\srdirr(\Lambda)$ such that $Z, Z'$ is a component cluster (not $E_{\Lambda}$-rigid). Then $g_{\Lambda}(Z')\in \mathcal{G}$ is given by $(-b,a,b-a)$ as can be seen from Theorem \ref{g}. Now, we can basically copy the proof of \cite[Prop.5.1.]{R1} to see, that this gives precisely the integer vectors as described in Proposition \ref{reading} (iii). The $\bf{g}$-vector of $\tau_{\Lambda'}Z$ is given by $(-1-b,a-1,b-a+1)$ as can be seen in Theorem \ref{gneu}. The statement on the compatibility follows from the symmetry of the graph. 
\end{proof}
The following theorem is a consequence of \cite[Theorem 4.4]{R3} in the special case of the once-punctured torus. In \cite{R3} this theorem is formulated in terms of \textit{integral quasi-laminations}, i.e. collections of pairwise compatible allowable curves with positive integer weights. 
\begin{theorem}
Every vector in $\mathbb{Z}^3$ is a positive $\mathbb{Z}$-linear combination of pairwise compatible vectors in $\mathcal{G}$. 
\end{theorem}
\begin{corollary}
Let $Z \in \dirr^{\sr}(\Lambda')$. Then either  $Z=Z'$ or $Z=\tau_{\Lambda'} Z'$ for some $Z' \in \srdirr(\Lambda)$.
\end{corollary}
\section{Markov Conjecture}
A \textit{Markov triple} is a triple $(a,b,c)$ consisting of positive integers, which satisfy the Diophantine Equation
\begin{equation*}
a^2+b^2+c^2=3abc.
\end{equation*}

Given a Markov triple, one can obtain a new Markov triple by mutation. More precisely: let $(a,b,c)$ be a Markov triple and set
\begin{eqnarray*}
a'=3bc-a, & b'=3ac-b, & c'=3ab-c.
\end{eqnarray*}
Then $(a',b,c), (a,b',c)$ and $(a,b,c')$ are Markov triples. Note that the mutation of a Markov triple is an involution. What is more, Markov proved in \cite{M} that any Markov triple can be obtained by repeatedly mutating the Markov triple $(1,1,1)$.

Frobenius was the first to claim that every Markov number, by definition a number appearing in a Markov triple, appears uniquely as a largest number of a Markov triple (up to permutation).
\begin{mcon}\cite{F}
Any Markov triple is uniquely determined by its largest entry (up to permutation).
\end{mcon} 
Recall, that by Laurent Phenomenon, any cluster variable $X$ can be written as a Laurent polynomial in a given initial cluster  $(x_1,x_2,x_3)$, so we have $X=X(x_1,x_2,x_3)$. Using the mutation of cluster variables one can show the following connection between the Markov numbers and the cluster algebra associated with the once-punctured torus $\mathcal{A}_Q$:
\begin{theorem}[\cite{PZ}]\label{bin}
There is a bijection between the set of all clusters in $\mathcal{A}_Q$ and the set of all Markov triples given by
\begin{eqnarray*}
(X_1,X_2,X_3) \longmapsto (X_1(1,1,1), X_2(1,1,1), X_3(1,1,1)).
\end{eqnarray*} 
\end{theorem}
The bijection between the Markov triples and the cluster variables $\mathcal{M}_Q$ of $\mathcal{A}_Q$ and their description via generic Caldero-Chapoton functions yield a new way of computing Markov numbers: 
Let $X \in \mathcal{M}_Q$ be a cluster variable, $Z \in \dirr(\Lambda)$ an $E_{\Lambda}$-rigid indecomposable strongly reduced component with $C_{\Lambda}(Z)=X$ and $M\in Z$ an $E_{\Lambda}$-rigid string module. Then by definition of the Caldero-Chapoton function the Markov number
\[X(1,1,1)=C_{\Lambda}(Z)(1,1,1)=\sum_{e\in \mathbb{N}^3} \chi(\text{Gr}_e(M))\] 
is given by the sum of Euler characteristics of Grassmannians of subrepresentations of $M$.
\begin{definition}
Let $C$ be a string given by the sequence $(x_1:a_1,a_2,\ldots ,a_n)$ and $M(C)$ the corresponding string module. We set
\[m(x_1:a_1,a_2,\ldots ,a_n):=m(M(C)):=\sum_{e\in \mathbb{N}^3} \chi(\text{Gr}_e(M(C))).\]
If $(x_1:a_1,a_2,\ldots ,a_n)$ describes an $E_{\Lambda}$-rigid string, then $m(x_1:a_1,a_2,\ldots ,a_n)$ is a Markov number.
\end{definition}
\begin{con}\label{con}
Let $C=(x_1:a_1,a_2,\ldots ,a_n)$ and $C'=(x'_1:a'_1,a'_2,\ldots ,a'_m)$ be two $E_{\Lambda}$-rigid strings. Then we have
\[m(C)=m(C')\]
if and only if $n=m$ and $a_i=a'_i$ for all $1\leq i \leq n$. 
\end{con}
\begin{remark}
By Proposition \ref{neighbours} it is easy to see, that an $E_{\Lambda}$-rigid $C$ string has precisely two neighbours previously denoted by $C_1$ and $C_2$ whose Markov numbers are smaller. Hence, if the Conjecture above is true, the Markov Conjecture follows. Therefore, one should try to find a closed formula for computing the Markov number of an $E_{\Lambda}$-rigid string $C$. In \cite{H} it was proven that in order to compute the Euler characteristics of quiver Grassmannians of string modules we have to count the successor closed subquivers of the diagram of the corresponding string. 
\end{remark}
\begin{lemma}\label{fib}
Denote by $F(n)$ the $n$-th Fibonacci number. We have
\[m(x_1:a)=F(2a+3)\]
for all $a \in \mathbb{Z}_{\geq -1}$ and all $x_1 \in \lbrace \alpha_1^{\pm}, \beta_1^{\pm},\gamma_1^{\pm} \rbrace$.
\end{lemma}
\begin{proof}
We only consider the case $x_1 \in Q_1$, the case $x_1 \in Q_1^-$ is proven similarly. Let $C_n$ be the string of length $n$ given by the sequence
\[c_1c_2 c_3 c_4 \ldots c_n =
\begin{cases}
x_1x_2^-x_1x_2^- \ldots x_1x_2^- &\text{ if } n  \equiv 0 \text{ mod }  2\\
x_1x_2^-x_1x_2^- \ldots x_1 &\text{ if } n  \equiv 1 \text{ mod }  2.
\end{cases}
\]
It is easy to see that $m(C_0)=2=F(3)$ and $m(C_1)=3=F(4)$. 

Now assume the claim holds for $n$ and consider $n+1$. We first assume that $n$ is even. Then the diagram of $C_{n+1}= c_1c_2 c_3 c_4 \ldots c_{n-1} c_n c_{n+1}$ is given by 
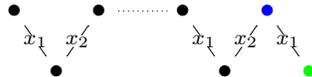
\[
\xygraph{!{<0cm,0cm>;<0.56cm,0cm>:<0cm,0.8cm>::}
!{(0,0)}*+{\bullet}="1"
!{(1,-1)}*+{\bullet}="2"
!{(2,0)}*+{\bullet}="3"
!{(4,0)}*+{\bullet}="4"
!{(5,-1)}*+{\bullet}="5"
!{(6,0)}*+[blue]{\bullet}="6"
!{(7,-1)}*+[green]{\bullet}="7"
"1"-"2"|-{x_1}
"2"-"3"|-{x_2}
"3":@{.}"4"
"4"-"5"|-{x_1}
"5"-"6"|-{x_2}
"6"-"7"|-{x_1}
}
\]
We first count the successor closed subquivers containing the basis vector at the green bullet. This number is given by $m(C_n)$: The successor closed subquivers of $C_n$ not containing the blue bullet are also successor closed subquivers of $C_{n+1}$ together with the green bullet. The successor closed subquivers of $C_n$ containing the blue bullet become successor closed subquivers of $C_{n+1}$ only by adding the green bullet. So by induction there are $m(C_n)=F(n+3)$ successor closed subquivers containing the basis vector at the green bullet.

Now we consider successor closed subquivers of $C_{n+1}$ without the green bullet. These cannot contain the blue bullet and thus all that remains are the successor closed subquivers in $C_{n-1}$ which by induction are given by $F(n+2)$. Thus we find $m(C_{n+1})=F(n+4)$. 

If $n$ is odd the proof is similar, where the diagram of $C_n$ is
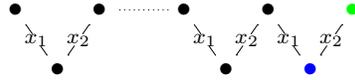
\[
\xygraph{!{<0cm,0cm>;<0.56cm,0cm>:<0cm,0.8cm>::}
!{(0,0)}*+{\bullet}="1"
!{(1,-1)}*+{\bullet}="2"
!{(2,0)}*+{\bullet}="3"
!{(4,0)}*+{\bullet}="4"
!{(5,-1)}*+{\bullet}="5"
!{(6,0)}*+{\bullet}="6"
!{(7,-1)}*+[blue]{\bullet}="7"
!{(8,0)}*+[green]{\bullet}="8"
"1"-"2"|-{x_1}
"2"-"3"|-{x_2}
"3":@{.}"4"
"4"-"5"|-{x_1}
"5"-"6"|-{x_2}
"6"-"7"|-{x_1}
"7"-"8"|-{x_2}
}
\]
and one counts again the successor closed subquivers containing the green bullet and the ones without the green bullet.
\end{proof}
\begin{prop}\label{div}
 We have
\[m(a_1,\ldots a_{i},\ldots,a_n)=m(a_1,\ldots a_i)m(a_{i+1},\ldots, a_n)+m(a_1,\ldots a_i -1)m(a_{i+1} -1,\ldots, a_n)\]
for all $1\leq i \leq n-1$.
\end{prop}
\begin{proof}
We consider the case $n=2$ and $i=1$,i.e. $(a_1,a_2)$, since the general case follows similarly.
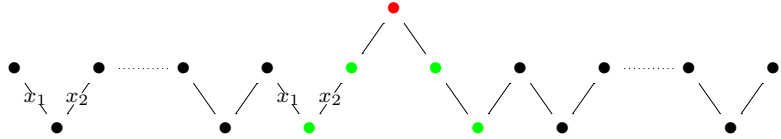
\[
\xygraph{!{<0cm,0cm>;<0.56cm,0cm>:<0cm,0.8cm>::}
!{(-2,0)}*+{\bullet}="31"
!{(-1,-1)}*+{\bullet}="1"
!{(0,0)}*+{\bullet}="2"
!{(2,0)}*+{\bullet}="3"
!{(3,-1)}*+{\bullet}="35"
!{(4,0)}*+{\bullet}="4"
!{(5,-1)}*+[green]{\bullet}="5"
!{(6,0)}*+[green]{\bullet}="6"
!{(7,1)}*+[red]{\bullet}="7"
!{(8,0)}*+[green]{\bullet}="8"
!{(9,-1)}*+[green]{\bullet}="9"
!{(10,0)}*+{\bullet}="13"
!{(11,-1)}*+{\bullet}="39"
!{(12,0)}*+{\bullet}="30"
!{(14,0)}*+{\bullet}="14"
!{(15,-1)}*+{\bullet}="15"
!{(16,0)}*+{\bullet}="16"
"1"-"31"|-{x_1}
"2"-"1"|-{x_2}
"4"-"5"|-{x_1}
"5"-"6"|-{x_2}
"2":@{.}"3"
"3"-"35"
"35"-"4"
"6"-"7"
"7"-"8"
"8"-"9"
"9"-"13"
"13"-"39"
"39"-"30"
"30":@{.}"14"
"14"-"15"
"15"-"16"
}
\]
First counting the successor closed subquivers without the red bullet, gives $m(a_1)m(a_2)$ successor closed subquivers. If the red bullet is contained in a successor closed subquiver, then so are the green bullets. Thus there are $m(a_1-1)m(a_2-1)$ successor closed subquivers containing the red bullet. The propositions follows. 
\end{proof}
\begin{corollary}
Let $(x_1:a_1,\ldots,a_n)$ describe a string. Then 
\[m(x_1:a_1,\ldots,a_n)\] does not depend on $x_1$.
\end{corollary}
\begin{proof}
This follows directly from Lemma \ref{fib} and Proposition \ref{div}.
\end{proof}
\begin{prop}
We have
\[m(a_1+\cdots +a_n + \lceil n/2\rceil -1) < m(a_1,\ldots,a_n)< m(a_1+\cdots +a_n + n -1)\]
for all $n\geq 2$.
\end{prop}
\begin{proof}
By Lemma \ref{fib} we have $m(a)=F(2a+3)$ and by Proposition \ref{div} we have\[m(a_1,a_2)=m(a_1)m(a_2)+m(a_1-1)m(a_2-1)=F(2a_1+3)F(2a_2+3)+F(2a_1+1)F(2a_2+1).\]
Using the well-known formula 
\[F(n+m)=F(n+1)F(m)+F(n)F(m-1)\]
for the Fibonacci numbers, we see that
\[m( a_1,a_2)< F(2a_1+3)F(2a_2+3)+F(2a_1+2)F(2a_2+2)=F(2a_1+2a_2+5)=m(a_1+a_2+1)\]
and 
\[m( a_1,a_2)> F(2a_1+2)F(2a_2+2)+F(2a_1+1)F(2a_2+1)=F(2a_1+2a_2+3)=m(a_1+a_2).\]
Now for $n\geq 2$ the propositions follows by induction, where the induction step, is dealt with in the same way as case $n=2$.
\end{proof}
When most of this work was done we learned of works by J. Propp \cite{Pr}, in which he describes an interpretation of Markov numbers as the number of perfect matchings in certain Snake graphs. In a more general note, we want to point out that the combinatorics of diagrams of strings of $\Lambda$ are very similar to the combinatorics of Snake graphs. 

A \textit{Snake graph with a sign function} $(G,f)$ was defined by Canacki and Schiffler \cite{CS}. We will recall a rather informal definition. The undirected graph $G$ is connected and consists of a finite sequence of tiles (i.e. squares) $G_1,\ldots, G_d$ with $d\geq 1$ which are glued along the interior edges $e_1,\ldots e_{d-1}$ according to the sign function $f\colon \lbrace 1,2,\ldots, d-1 \rbrace \rightarrow \pm 1$. Where the sign function alternates, the snake graph is straight, i.e. the corresponding tiles lie in one column or one row. When the sign function is constant, the snake graph is zigzag, i.e. no three consecutive tiles are straight.

A \textit{string diagram with sign function} $(C,f)$ is a connected, undirected graph $C$ (of an underlying string diagram) consisting of a finite sequence of vertices $v_1,\ldots v_d$ where $d\geq $ which are connected by edges $c_1,\ldots c_{d-1}$ according to the sign function $f\colon \lbrace 1,2,\ldots, d-1 \rbrace \rightarrow \pm 1$. Where the sign function alternates, the string diagram is zigzag, i.e. the edges alternate between $Q_1$ and $Q_1^-$. When the sign function is constant, the string diagram is straight, i.e. all corresponding edges belong to either $Q_1$ or $Q_1^-$. The following proposition is proven by an easy induction.
\begin{prop}
Let $(G,f)$ be a snake graph and $(C,f)$ a string diagram given by the same sign function $f$. Then the number of perfect matchings in $G$ is $m(C)$. 
\end{prop}


\end{document}